\documentclass[12pt]{amsart}
\usepackage{a4wide,enumerate,color,graphicx}
\usepackage{enumitem}
\usepackage{bm} 
\usepackage{scrextend} 
\allowdisplaybreaks

\usepackage{algorithm}
\usepackage{algpseudocode}

\let\pa\partial  
\let\na\nabla  
\let\eps\varepsilon  
\newcommand{\N}{{\mathbb N}}  
\newcommand{\R}{{\mathbb R}} 
\newcommand{\diver}{\operatorname{div}}

\newtheorem{theorem}{Theorem}   
\newtheorem{lemma}[theorem]{Lemma}   
   
\newtheorem{remark}[theorem]{Remark}   
\newtheorem{corollary}[theorem]{Corollary}

\begin{document}  

\title[A Galerkin scheme for Poisson--Maxwell--Stefan systems]{Convergence of an 
implicit Euler Galerkin scheme \\ for Poisson--Maxwell--Stefan systems}

\author[A. J\"ungel]{Ansgar J\"ungel}
\address{Institute for Analysis and Scientific Computing, Vienna University of  
	Technology, Wiedner Hauptstra\ss e 8--10, 1040 Wien, Austria}
\email{juengel@tuwien.ac.at} 

\author[O. Leingang]{Oliver Leingang}
\address{Institute for Analysis and Scientific Computing, Vienna University of  
	Technology, Wiedner Hauptstra\ss e 8--10, 1040 Wien, Austria}
\email{oliver.leingang@tuwien.ac.at} 

\date{\today}

\thanks{The authors acknowledge partial support from   
the Austrian Science Fund (FWF), P27352, P30000, F65, and W1245}

\begin{abstract}
A fully discrete Galerkin scheme for a thermodynamically consistent transient
Max\-well--Stefan system for the mass particle densities, coupled to the 
Poisson equation for the electric potential, is investigated.
The system models the diffusive dynamics of an isothermal ionized fluid mixture
with vanishing barycentric velocity.
The equations are studied in a bounded domain, and different molar masses 
are allowed. The Galerkin scheme preserves the total mass,
the nonnegativity of the particle densities,
their boundedness, and satisfies the second law of thermodynamics 
in the sense that the discrete entropy production is nonnegative.
The existence of solutions to the Galerkin scheme and the convergence of a subsequence
to a solution to the continuous system is proved. Compared to previous works, the 
novelty consists in the treatment of the drift terms involving the electric field.
Numerical experiments show the sensitive dependence of the particle densities
and the equilibration rate on the molar masses.
\end{abstract}

 

\maketitle


\section{Introduction}

The Maxwell--Stefan equations describe the dynamics of a fluid mixture in the
diffusive regime. They have numerous applications, for instance, 
in sedimentation, dialysis, electrolysis, and ion exchange. 
While Maxwell--Stefan models have been
investigated since several decades from a modeling and simulation viewpoint 
in the engineering literature (e.g.\ \cite{GaMa92}),
the mathematical and numerical analysis started more recently \cite{Bot11,GiMa98}.
The global existence of weak solutions under natural conditions was proved in
\cite{ChJu15,JuSt13} for neutral mixtures. In case of ion transport, the
electric charges and the self-consistent electric potential need to be taken into
account. To our knowledge, no mathematical results are available in the
literature for such Poisson--Maxwell--Stefan models. In this paper, we prove the
existence of a weak solution to a structure-preserving fully discrete Galerkin 
scheme and its convergence to the continuous problem. This provides, for the first time,
a global existence result for Poisson--Maxwell--Stefan systems.

\subsection{Model equations}

We consider an ionized fluid mixture consisting of $n$ components with the partial 
mass density $\rho_i$, partial flux $J_i$, and molar mass $M_i$ of the $i$th species.
The evolution of the particle densities $\rho_i$ is governed by the partial mass
balance equations
\begin{equation}\label{1.rhoi}
  \pa_t\rho_i + \diver J_i = r_i(x), \quad i=1,\ldots,N,
\end{equation}
where $r_i$ are the production rates satisfying $\sum_{i=1}^n r_i(x)=0$ and
$\sum_{i=1}^n J_i=0$. 
The molar concentrations are defined by $c_i=\rho_i/M_i$ and $x_i=c_i/c$ are the 
molar fractions, where $c_{\rm tot}=\sum_{i=1}^n c_i$ denotes the total concentration
and we have set $x=(x_1,\ldots,x_n)$.
The partial fluxes $J_i$ and the gradients of the molar fractions $x_i$ are related by 
the (scaled) Maxwell--Stefan equations
\begin{equation}\label{1.Ji}
  -\sum_{j=1}^N k_{ij}(\rho_jJ_i-\rho_iJ_j)
	= D_i := \na x_i + (z_ix_i - (\rho\cdot x)\rho_i)\na\Phi, \quad i=1,\ldots,n.
\end{equation}
where $k_{ij}=k_{ji}$ are the rescaled (reciprocal) Maxwell--Stefan diffusivities,
$D_i$ is the driving force, $z_i$ the electric charge of the $i$th component, 
and $\Phi$ the electric potential. We refer to Section \ref{sec.model} for
details on the modeling.
These equations are coupled to the (scaled) Poisson equation 
\begin{equation}\label{1.phi}
  -\lambda\Delta\Phi = \sum_{i=1}^n z_ic_i + f(y),
\end{equation}
where $\lambda$ is the scaled permittivity and $f(y)$ is a fixed background charge.
The equations are solved in a bounded domain $\Omega\subset\R^d$ ($d\ge 1$)
and supplemented by the boundary conditions
\begin{align}
  & J_i\cdot\nu = 0 \quad\mbox{on }\pa\Omega,\ i=1,\ldots,n, \label{1.bc.Ji} \\
	& \Phi=\Phi_D \quad\mbox{on }\Gamma_{\rm D}, \quad \na\Phi\cdot\nu=0\quad\mbox{on }
	\Gamma_{\rm N}, \label{1.bc.phi}
\end{align}
where $\Gamma_{\rm D}$ models the electric contacts, 
$\Gamma_N=\pa\Omega\backslash\Gamma_{\rm D}$ is the union of insulating 
boundary segments, and $\nu$ denotes the exterior unit normal vector
to $\pa\Omega$. This means that the mixture cannot leave the container $\Omega$
and an electric field is applied at the contacts $\Gamma_{\rm N}$.  
The initial conditions are given by
\begin{equation}\label{1.ic}
  \rho_i(\cdot,0) = \rho_i^0\quad\mbox{in }\Omega, \quad i=1,\ldots,n.
\end{equation}
We assume that the total mass is constant initially, $\sum_{i=1}^n\rho_i^0=1$,
which implies from \eqref{1.rhoi} that the total mass is constant for all
times, $\sum_{i=1}^n\rho_i(t)=1$, expressing total mass conservation.

Observe that \eqref{1.Ji} defines a linear system in the diffusion fluxes.
Since $\sum_{i=1}^n D_i=0$, the kernel of that system is nontrivial, and we
need to invert the relation between the fluxes $J_i$ and the driving forces $D_i$
on the orthogonal component of the kernel. It was shown in \cite[Section 2]{JuSt13}
that we can write \eqref{1.Ji} as $D'=-A_0J'$,
where $D'=(D_1,\ldots,D_{n-1})$, $J'=(J_1,\ldots,J_{n-1})$, 
and $A_0\in\R^{(n-1)\times(n-1)}$ is invertible; see Section \ref{sec.flux}
for details.
The $n$th components are recovered from $D_n=-\sum_{i=1}^{n-1}D_i$ and 
$J_n=-\sum_{i=1}^{n-1}J_i$. Thus, \eqref{1.rhoi} can be written compactly as
the cross-diffusion system \cite{Bot11,JuSt13}
$$
  \pa_t\rho' - \diver(A_0^{-1}D') = r'(x),
$$
where $\rho'=(\rho_1,\ldots,\rho_{n-1})$. However, $A_0^{-1}$ is not positive
definite. To obtain a positive definite diffusion matrix, we need to transform
the system. With the so-called entropy variables
\begin{equation}\label{1.wi}
  w_i = \frac{\log x_i}{M_i} - \frac{\log x_n}{M_n} 
	+ \bigg(\frac{z_i}{M_i}-\frac{z_n}{M_n}\bigg)\Phi, \quad i=1,\ldots,n-1,
\end{equation}
we may formulate \eqref{1.rhoi} as 
\begin{equation}\label{1.B}
  \pa_t\rho' - \diver(B\na w) = r'(x),
\end{equation}
where $B=(B_{ij})\in\R^{(n-1)\times(n-1)}$ is symmetric and positive definite;
see Section \ref{sec.flux} for details. Here, $\rho'$ and $x$ are interpreted
as (invertible) functions of $w$ and $\Phi$. 
This transformation is well known in nonequilibrium thermodynamics,
where $w_i$ is called the electro-chemical potential and $B$ is the mobility or
Onsager matrix. 

The transformation to entropy variables has two important advantages. First,
introducing the entropy
\begin{equation}\label{1.H}
  H(\rho) = \int_\Omega h(\rho)dy, \quad
	h(\rho) = c_{\rm tot}\sum_{i=1}^n x_i\log x_i 
	+ \frac{\lambda}{2}|\na(\Phi-\Phi_D)|^2,
\end{equation}
a formal computation shows that
\begin{equation}\label{1.epi}
  \frac{dH}{dt} + \int_\Omega\na w:B\na w dy
	= \int_\Omega\sum_{i=1}^n r_i(x)\frac{\pa h}{\pa\rho_i}dy,
\end{equation}
if $\Phi_D$ is constant, where $A:B$ denotes the Frobenius matrix product
between matrices $A$ and $B$. 
(A discrete analog is shown in Theorem \ref{thm.ex} below.)
Thus, if the right-hand side is nonpositive, 
the entropy $t\mapsto H(\rho(t))$ is a Lyapunov functional and
we may obtain suitable estimates for $w_i$. The entropy production
(the diffusion term) is nonnegative, which expresses the second law of
thermodynamics.
This technique has been used in \cite{ChJu15,JuSt13} but without electric force terms.
The derivation of gradient estimates is more delicate in the presence of
the electric potential; see Lemma \ref{lem.diff}. Second, the densities
$\rho_i=\rho_i(w)$ are automatically positive and bounded and it holds that
$\sum_{i=1}^n\rho_i(w)=1$; see Corollary \ref{coro.inv}. This property is inherent
of the transformation and it holds without the use of a maximum principle
and independent of the functional setting.

The aim of this paper is to extend the global existence result of 
\cite{ChJu15,JuSt13} to Maxwell--Stefan systems with electric forces and to 
suggest a fully discrete Galerkin scheme that preserves the structure of the
system, namely the nonnegativity of the particle densities, the $L^\infty$
bound $\sum_{i=1}^n\rho_i=1$, and a discrete analog of the entropy production
inequality \eqref{1.epi}.

\subsection{State of the art}

Before presenting our main results, we briefly review the state of the art
of Maxwell--Stefan models. They were already derived in the 19th century
by Maxwell using kinetic gas theory \cite{Max66} and Stefan using 
continuum mechanics \cite{Ste71}. A more mathematical derivation from the
Boltzmann equation can be found in \cite{BGPS13,Gio99}, including a non-isothermal 
setting \cite{HuSa17}. An advantage of the Maxwell--Stefan approach
is that the definition of the driving forces can be adapted to the
present physical situation, leading to very general and thermodynamically
consistent models \cite{BoDr15}. 

When electrolytes are considered, we need to
take into account the electric force. Usually, this is done in the context
of Nernst--Planck models \cite{Ner89,Pla90}, where the diffusion flux $J_i$
only depends on the density gradient of the $i$th component, thus without any 
cross-diffusion effects. Duncan and Toor \cite{DuTo62} showed that cross-diffusion
terms need to be taken into account in a ternary gas. Dreyer et al.\ 
\cite{DGM13} outline some deficiencies of Nernst--Planck models and 
propose thermodynamically consistent Maxwell--Stefan type models. A numerical
comparison between Nernst--Planck and Maxwell--Stefan models can be found in
\cite{PsFa11}.

The first global-in-time existence result to the Maxwell--Stefan equations
\eqref{1.rhoi}-\eqref{1.Ji} without Poisson equation was proved by Giovangigli and
Massot \cite{GiMa98} for initial data around the constant equilibrium state.
The local-in-time existence of classical solutions was shown by Bothe \cite{Bot11}. 
The entropy structure of the Maxwell--Stefan system was revealed in
\cite{JuSt13}, and a general global existence theorem could be shown.
Further global existence results can be found in \cite{HMPW17,MaTe15}.
The Maxwell--Stefan system was coupled to the heat equation \cite{HuSa18}
and to the incompressible Navier--Stokes equations \cite{ChJu15}.
In \cite[Theorem 9.7.4]{Gio99} and \cite[Theorem 4.3]{HMPW17}, 
the large-time asymptotics for initial data close to equilibrium was analyzed.
The convergence to equilibrium for any initial data was investigated in
\cite{ChJu15,JuSt13} without production terms and in \cite{DJT18}
with production terms for reversible reactions. Salvarani and Soares
proved a relaxation limit of the Maxwell--Stefan system to a system of
linear heat equations \cite{SaSo18}.

Surprisingly, there are not many papers concerned with numerical schemes which
preserve the properties of the solution like conservation of total mass,
nonnegativity, and entropy production. Many approximation schemes can be found in
the engineering literature, for instance finite-difference \cite{LeAn10,LVM92}
or finite-element \cite{CaCa08} discretizations. In the mathematical literature,
finite-volume \cite{PDBLGM11} and mixed finite-element \cite{McBo14} schemes 
as well as explicit finite-difference schemes with fast solvers \cite{Gei15}
were proposed.
The existence of discrete solutions was shown in \cite{McBo14}, but only for
ternary systems and under restrictions on the diffusion coefficients.
The schemes of \cite{BGS12,PDBLGM11} conserve the total mass, while those
of \cite{BGS12,DMB15} also preserve the $L^\infty$ bounds. The result of
\cite{DMB15} is based on maximum principle arguments. Note that we are able
to show the $L^\infty$ bounds without the use of a maximum principle, as a
result of the formulation in terms of entropy variables, and that we do
not impose any restrictions on the diffusivities (except positivity).

All the cited results are concerned with the Maxwell--Stefan equations for
neutral fluids, i.e.\ without electric effects. In this paper, we analyze for
the first time Poisson--Maxwell--Stefan systems and show a discrete entropy
production inequality. The cross-diffusion terms cause some mathematical
difficulties which are not present in Nernst--Planck models.

\subsection{Main results}\label{sec.main}

Let $(\theta^{(k)})$ be an orthonormal basis of $H_D^1(\Omega)$  
and $(v^{(k)})$ be an orthonormal basis of $H^1(\Omega;\R^{n-1})$ such that
$v^{(k)}\in L^\infty(\Omega;\R^{n-1})$. 
We introduce the Galerkin spaces
$$
  P_N = \operatorname{span}\{u^{(1)},\ldots,u^{(N)}\}, \quad
	V_N = \operatorname{span}\{v^{(1)},\ldots,v^{(N)}\}.
$$
Furthermore, let $T>0$ and $N\in\N$ and set $\tau=T/N>0$.
We impose the following assumptions:

\begin{labeling}{(A1)}
\item[(A1)] Domain: $\Omega\subset\R^d$ is a bounded domain with Lipschitz boundary
$\pa\Omega=\Gamma_{\rm D}\cup\Gamma_{\rm N}$,
where $\Gamma_{\rm D}\cap\Gamma_{\rm N}=\emptyset$, $\Gamma_{\rm N}$ is open 
in $\pa\Omega$, and $\text{meas}(\Gamma_{\rm D})>0$.

\item[(A2)] Given functions: The initial datum 
$\rho^0=(\rho_1^0,\ldots,\rho_n^0)$ is nonnegative and
measurable satisfying $\int_\Omega\sum_{i=1}^n\rho_i\log\rho_i dy<\infty$, 
$\rho_n^0=1-\sum_{i=1}^{n-1}\rho_i^0\ge 0$. The boundary data $\Phi_D\in H^1(\Omega)
\cap L^\infty(\Omega)$ solves $-\lambda\Delta\Phi_D=f$ in $\Omega$ and
$\na\Phi_D\cdot\nu=0$ on $\Gamma_{\rm N}$. Furthermore, let $f\in L^\infty(\Omega)$.

\item[(A3)] Diffusion matrix: For any given $\rho\in[0,\infty)^n$, 
the transpose of the matrix $A=(A_{ij})\in\R^{n\times n}$, defined by 
\begin{equation}\label{1.A}
  A_{ij} = \left\{\begin{array}{ll}
	\sum_{\ell=1,\,\ell\neq i}^n k_{i\ell}\rho_\ell &\quad\mbox{for }i=j, \\
	-k_{ij}\rho_i &\quad\mbox{for }i\neq j,
	\end{array}\right.
\end{equation}
has the kernel $\operatorname{ker}(A^\top)=\operatorname{span}\{\bm{1}\}$,
where $\bm{1}=(1,\ldots,1)\in\R^n$. 

\item[(A4)] Production rates: The functions $r_i\in C^0([0,1]^n;\R)$
satisfy $\sum_{i=1}^n r_i(x)\log x_i/M_i\le 0$ for all $x\in(0,1]^n$,
$i=1,\ldots,n$.
\end{labeling}

Assumptions (A1) and (A2) are rather natural. 
The condition $\rho_i\log\rho_i\in L^1(\Omega)$
is needed to apply the entropy method. By definition of $A$,
it holds that $\operatorname{ker}(A^\top)
\subset\operatorname{span}\{\bm{1}\}$. If $k_{ij}>0$ (and $\rho_j>0$), a computation
shows that $\operatorname{span}\{\bm{1}\}=\operatorname{ker}(A^\top)$. For the
general case $k_{ij}\ge 0$, this property cannot be guaranteed and needs to
be assumed. This explains Assumption (A3).
Assumption (A4) is needed to derive the entropy production inequality \eqref{1.epi}.
It is satisfied for reversible reactions; see \cite[Lemma 6]{DJT18}.

We consider the implicit Euler Galerkin scheme
\begin{align}
  & \frac{1}{\tau}\int_\Omega\big(\rho'(u^k+w_D,\Phi^k) 
	- \rho'(u^{k-1}+w_D,\Phi^{k-1})\big)\cdot\phi dy 
	+ \eps\int_\Omega u^k\cdot\phi dy \nonumber \\
	&\phantom{xxxx}{}+ \int_\Omega\na\phi:B(u^k+w_D,\Phi^k)\na(u^k+w_D)dy
	= \int_\Omega r'(x(u^k+w_D,\Phi^k))\cdot\phi dy, \label{1.sch1} \\
	& \lambda\int_\Omega\na\Phi^k\cdot\na\theta dy
	= \int_\Omega\bigg(\sum_{i=1}^n z_ic_i(u^k+w_D,\Phi^k) + f(y)\bigg)dy \label{1.sch2}
\end{align}
for $\phi\in V_N$, $\theta\in P_N$, $\eps>0$, and we have defined
\begin{equation}\label{1.wD}
  w_D=(w_{D,1},\ldots,w_{D,n-1}), \quad
	w_{D,i} = \bigg(\frac{z_i}{M_i}-\frac{z_n}{M_n}\bigg)\Phi_D.
\end{equation}
The discrete entropy variables are given by $w^k = u^k + w_D$, 
and we used the notation
$c_i(w^k,\Phi^k)=\rho_i(w^k,\Phi^k)/M_i$, 
$x_i(w^k,\Phi^k)=c_i(w^k,\Phi^k)/c_{\rm tot}^k$ 
for $i=1,\ldots,n$, and $c_{\rm tot}^k=\sum_{i=1}^n\rho_i(w^k,\Phi^k)/M_i$.

At time $k=0$, we assume that $\rho_i^0\ge\eta>0$ in $\Omega$. This allows us to
define $w^0$ via definition \eqref{1.wi}. The condition can be removed by performing
the limit $\eta\to 0$ in the proof; see \cite{ChJu15} for details.
Furthermore, let $\Phi^0\in H^1(\Omega)\cap 
L^\infty(\Omega)$ be the unique solution to 
$$
  -\lambda\Delta\Phi^0 = \sum_{i=1}^n z_i\frac{\rho_i^0}{M_i} + f(y)
	\ \mbox{in }\Omega, \quad \na\Phi^0\cdot\nu=0\ \mbox{on }\Gamma_{\rm N}, \quad
	\Phi^0=\Phi_D\ \mbox{on }\Gamma_{\rm D}.
$$
This defines $(w^0,\Phi^0)$.

\begin{theorem}[Existence for the Galerkin scheme]\label{thm.ex}
Let Assumptions (A1)-(A4) hold.
Then there exists a weak solution $(w^k,\Phi^k)\in V_N\times P_N$ to 
\eqref{1.sch1}-\eqref{1.sch2} with $w^k=u^k+w_D$, satisfying 
\begin{itemize}
\item preservation of $L^\infty$ bounds: $0<\rho_i^k<1$ for $i=1,\ldots,n$;
\item conservation of total mass: $\sum_{i=1}^n\rho_i^k=1$ in $\Omega$;
\item discrete entropy production inequality:
\begin{align}
  H(\rho^k) &+ \tau\int_\Omega\na(w^k-w_D):B(w^k,\Phi^k)\na w^k dy 
	+ \eps\tau\int_\Omega|w^k-w_D|^2 dy \nonumber \\
  &\le \tau\int_\Omega\sum_{i=1}^n\frac{z_i}{M_i} r_i(x^k)(\Phi^k-\Phi_D)dy
	+ H(\rho^{k-1}), \label{1.ei}
\end{align}
where $\rho^k=\rho(w^k,\Phi^k)$. 
\end{itemize}
\end{theorem}

Theorem \ref{thm.ex} is proved by using a fixed-point argument in the entropy variables. 
Using $w^k-w_D$ as a test function in the fully discrete version of \eqref{1.B}, 
we show in Section \ref{sec.ex} that
$$
  H(\rho^k) + \tau K\int_\Omega\sum_{i=1}^n|\na (x_i^k)^{1/2}|^2dy 
	+ \eps\tau\int_\Omega|w^k-w_D|^2dy \le \tau K + H(\rho^{k-1}),
$$
where $K>0$ only depends on the given data. This is an estimated version of
\eqref{1.epi}. The term involving $\eps$ is needed to conclude a uniform
$L^2$ estimate for $w^k$, which is sufficient to apply the Leray-Schauder 
fixed-point theorem in the finite-dimensional Galerkin space. The $\eps$-independent
gradient estimate for $x_i^k$ cannot be used since it does not give an estimate
for $w_i^k$ (see \eqref{1.wi}).
It is possible to analyze system \eqref{1.sch1}-\eqref{1.sch2}
for $\eps=0$ -- see Step 2 of the proof of Theorem \ref{thm.conv} --, but we lose
the information about $w^k$ and obtain a solution in terms of $\rho^k$.
The term involving $\eps$ is technical and not essential for the numerical
simulations (or the structure preservation). 
However, we are not able to prove an existence
result in terms of the entropy variable without such a regularization.

\begin{remark}[Conservation of partial mass]\rm
When $r_i=0$, we have from \eqref{1.rhoi} conservation of the partial mass 
$\|\rho_i\|_{L^1(\Omega)}$. This conservation property does not hold exactly on
the discrete level because of the $\eps$-regularization. It holds that for any
$\delta>0$, there exists $\eps_0>0$ such that for any $0<\eps<\eps_0$
($\eps$ is the value in \eqref{1.sch1}),
\begin{align*}
  \big|\|\rho_i^k\|_{L^1(\Omega)} - \|\rho_i^0\|_{L^1(\Omega)}\big|
	&\le \delta\|\rho_i^0\|_{L^1(\Omega)}, \quad i=1,\ldots,n-1, \\
	\big|\|\rho^k_n\|_{L^1(\Omega)} - \|\rho_n^0\|_{L^1(\Omega)}\big|
	&\le \delta\sum_{i=1}^{n-1}\|\rho_i^0\|_{L^1(\Omega)}.
\end{align*}
The proof is the same as in \cite[Theorem 4.1]{JuSt13}.
As $\delta>0$ can be chosen arbitrarily small, this shows that the numerical
scheme preverses the partial mass approximately.
\qed
\end{remark}

\begin{theorem}[Convergence of the Galerkin solution]\label{thm.conv}\
Let Assumptions (A1)-(A4) hold.
Let $(\rho^k,\Phi^k)$ be a solution to \eqref{1.sch1}-\eqref{1.sch2} and set
$$
  \rho^\tau_i(y,t) = \rho_i^k(y), \quad x_i^\tau(y,t) = x_i^k(y),
	\quad c_i^\tau(y,t) = c_i^k(y), \quad \Phi^\tau(y,t) = \Phi^k(y)
$$
for $y\in\Omega$, $t\in((k-1)\tau,k\tau]$, $i=1,\ldots,n$ and introduce
the shift operator $(\sigma_\tau\rho_i^\tau)(y,t)=\rho_i^{k-1}(y)$ for $y\in\Omega$
and $t\in((k-1)\tau,k\tau]$. Then there exist subsequences (not relabeled) such
that, as $\eps\to 0$, $N\to\infty$, and $\tau\to 0$,
\begin{align*}
  \rho_i^\tau\to\rho_i &\quad\mbox{strongly in }L^p(0,T;L^p(\Omega))\mbox{ for any }
	p<\infty, \\
	x_i^\tau\rightharpoonup x_i,\quad \Phi^\tau\rightharpoonup\Phi 
	&\quad\mbox{weakly in }L^2(0,T;H^1(\Omega)), \\
  \tau^{-1}(\rho_i^\tau-\sigma_\tau(\rho_i^\tau))\rightharpoonup \pa_t\rho
	&\quad\mbox{weakly in }L^2(0,T;H^1(\Omega)'),\ i=1,\ldots,n,
\end{align*}
and the limit $(\rho,\Phi)$ satisfies for all $\phi\in L^2(0,T;H^1(\Omega;\R^{n-1}))$
and $\theta\in H_D^1(\Omega)$,
\begin{align}
  \int_0^T\langle\pa_t\rho',\phi\rangle dt + \int_0^T\int_\Omega\na\phi:A_0^{-1}(\rho)
	D' dydt &= \int_0^T\int_\Omega r'(x)\cdot\phi dydt, \label{1.weak1} \\
  \lambda\int_\Omega\na\Phi\cdot\na\theta dy 
	&= \int_\Omega\bigg(\sum_{i=1}^n z_i\frac{\rho_i}{M_i} 
	+ f(y)\bigg)\theta dy, \label{1.weak2}
\end{align}
where $D_i = \na x_i + (z_ix_i - (z\cdot x)\rho_i)\na\Phi$,
$\rho_i = c_{\rm tot}M_ix_i$, and $c_{\rm tot} = \sum_{i=1}^n\rho_i/M_i$.
Moreover, $\rho_n=1-\sum_{i=1}^{n-1}\rho_i$.
\end{theorem}

In Theorem \ref{thm.conv}, $\langle\cdot,\cdot\rangle$ denotes the duality bracket
between $H^1(\Omega;\R^{n-1})'$ and $H^1(\Omega;\R^{n-1})$.
The difficult part of the proof is the estimate
of the diffusion term because of the contribution of the electric field.
We show in Lemma \ref{lem.diff} that
$$
  \int_\Omega\na w^k:B\na w^k dy 
	\ge K\int_\Omega\sum_{i=1}^n M_i^{1/2}\frac{|D_i^k|^2}{x_i^k}dy
	\ge K_1\int_\Omega\sum_{i=1}^n|\na(x_i^k)^{1/2}|^2 dy - K_2
$$
holds for some constants $K$, $K_1$, $K_2>0$, which are independent of $\eps$, $N$,
and $\tau$. Then the uniform $L^\infty$ bound for $x_i^k$ gives a uniform
$H^1(\Omega)$ bound for $x_i^k$ and consequently for $\rho_i^k$. Weak compactness
allows us to pass to the limits $\eps\to 0$ and $N\to\infty$, and the
limit $\tau\to 0$ is performed by means of the Aubin-Lions lemma.

The paper is organized as follows. In Section \ref{sec.model}, we detail the
thermodynamic modeling of system \eqref{1.rhoi}-\eqref{1.phi}. Some
auxiliary results on the formulation of the fluxes $J_i$ and the inversion of
the map $\rho\mapsto w$ are presented in Section \ref{sec.aux}. 
Sections \ref{sec.ex} and \ref{sec.conv} are devoted to the proof of the main 
theorems. Finally, some numerical experiments are shown in Section \ref{sec.numer}.


\section{Modeling}\label{sec.model}

We consider an isothermal electrolytic mixture of $n$ fluid components 
in the bounded domain $\Omega\subset\R^d$ ($d\ge 1$) with boundary $\pa\Omega$. 
We assume that the
mixture is not moving, so the barycentric velocity vanishes. The thermodynamic
state of the mixture is described by the partial mass densities 
$\rho_1,\ldots,\rho_n$ and the electric field $E$. We suppose the
quasi-static approximation $E=-\na\Phi$, where $\Phi$ is the electric potential. 
The evolution of the mass densities $\rho_i=M_ic_i$ with the molar masses $M_i$ and
molar concentrations (or number densities) $c_i$
is governed by the partial mass balances \cite[(4)]{DDGG17}
$$
  \pa_t\rho_i + \diver J_i = r_i(x)\quad\mbox{in }\Omega,\ t>0,\ i=1,\ldots,n,
$$
where $x=(x_1,\ldots,x_n)$ is the vector of molar fractions 
$x_i=\rho_i/(c_{\rm tot}M_i)$, $c_{\rm tot}=\sum_{i=1}^n c_i$ is the 
total concentration, $J_i$ the diffusion 
flux, and $r_i(x)$ the mass production rate of the $i$th species. 
We assume that the total flux and the total production vanishes,
$$
  \sum_{i=1}^n J_i=0, \quad \sum_{i=1}^n r_i(x) = 0,
$$
which are necessary constraints to achieve total mass conservation,
$\pa_t\sum_{i=1}^n\rho_i=0$. We suppose that the total initial mass is constant
in space, $\sum_{i=1}^n\rho_i^0=\rho_{\rm tot}>0$, 
which implies that the total mass is constant in space and time,
$\sum_{i=1}^n\rho_i(t)=\rho_{\rm tot}$ for $t>0$. 

The electric potential $\Phi$ is given by the Poisson equation
\cite[(3) and (25)]{DGM13}
$$
  -\eps_0(1+\chi)\Delta\Phi = F\sum_{i=1}^n z_ic_i + f(y)\quad\mbox{in }\Omega,
$$
where $\eps_0$ is the dielectric constant, $\chi$ the dielectric susceptibility, 
$F$ the Faraday constant, $z_i$ the charge number of the $i$th species, 
and $f(y)$ with $y\in\Omega$ models the charge of fixed background ions. 

The basic assumption of the Maxwell--Stefan theory is that the difference
in speed and molar fractions leads to a diffusion flux. 
They are implicitly given by the driving forces $d_i$ according to 
\cite[(200)]{BoDr15}
$$
  -\sum_{j=1}^n\frac{x_j(J_i/M_i)-x_i(J_j/M_j)}{c_{\rm tot}D_{ij}} 
	= d_i, \quad i=1,\ldots,n,
$$
where the numbers
$D_{ij}=D_{ji}$ are the Maxwell--Stefan diffusivities. Inserting the definition
$x_i=\rho_i/(c_{\rm tot}M_i)$, we find that
\begin{equation}\label{2.MSE}
  -\sum_{j=1}^n\frac{\rho_jJ_i-\rho_iJ_j}{c_{\rm tot}^2M_iM_jD_{ij}} = d_i.
\end{equation}
In the present situation, the driving force is given by two components,
the variation of the chemical potential $\mu_i$ and the contribution of the
body forces $b_i$ \cite[(211)]{BoDr15}:
$$
  d_i = \frac{c_iM_i}{RT}\na\mu_i - \frac{\rho_i}{RT}(b_i-b_{\rm tot}), 
	\quad i=1,\ldots,n,
$$
where $R$ is the gas constant and $T$ the (constant) temperature.
Since $(D_{ij})$ is symmetric, summing \eqref{2.MSE} from $i=1,\ldots,n$ leads to
$\sum_{i=1}^n d_i=0$. Furthermore, $\sum_{i=1}^n\na\mu_i$ vanishes too; see below.
This shows that $b_{\rm tot}=\rho_{\rm tot}^{-1}\sum_{i=1}^n \rho_ib_i$. 
We assume that the only force is
due to the electric field (i.e., we neglect effects of gravity),
$b_i=-(z_i/M_i)F\na\Phi$ \cite[(3)]{PsFa11}.

It remains to determine the chemical potential. We define it by 
$\mu_i=\pa h_{\rm mix}/\pa\rho_i$, where 
$h_{\rm mix}(\rho) = c_{\rm tot}RT(\sum_{i=1}^n x_i\log x_i+1)$ is the 
mixing free energy density \cite[(23)]{DDGG17}. Then
$$
  \mu_i = \frac{1}{c_{\rm tot}M_i}\frac{\pa h_{\rm mix}}{\pa x_i} 
	= \frac{RT}{M_i}(\log x_i+1),
$$
and the driving force becomes
\begin{align}
  d_i &= c_i\na\log x_i + \frac{\rho_i F}{RTM_i}
	\bigg(z_i - \frac{1}{\rho_{\rm tot}}\sum_{j=1}^n\frac{z_j\rho_j}{M_j}\bigg)\na\Phi 
	\nonumber \\
	&= c_{\rm tot}\bigg(\na x_i + \frac{F}{RT}\bigg(z_ix_i - (z\cdot x)
	\frac{\rho_i}{\rho_{\rm tot}}\bigg)\na\Phi\bigg), \label{2.di}
\end{align}
where $z=(z_1,\ldots,z_n)$ and $x=(x_1,\ldots,x_n)$. The Gibbs-Duhem equation
$$
  \sum_{i=1}^n\rho_i\frac{\pa h_{\rm mix}}{\pa\rho_i} - h_{\rm mix}(\rho)
	= RT\sum_{i=1}^n\rho_i\frac{\log x_i+1}{M_i} 
	- c_{\rm tot}RT\bigg(\sum_{i=1}^n x_i\log x_i+1\bigg) = 0
$$
shows that the pressure vanishes, which is consistent with our choice of
the driving force (see \cite[(211)]{BoDr15}). 
The driving force in \cite[(7)]{PsFa11} contains a non-vanishing pressure
that is related to our expression for the total body force. The resulting
driving force \eqref{2.di}, however, is the same.

We summarize the model equations:
\begin{align}
  \pa_t\rho_i + \diver J_i &= r_i(x), \quad i=1,\ldots,n, \label{2.rhoi} \\
	-\eps_0(1+\chi)\Delta\Phi &= F\sum_{i=1}^n z_ic_i + f(y), \label{2.phi} \\
  -\sum_{j=1}^n\frac{\rho_jJ_i-\rho_iJ_j}{c_{\rm tot}^3M_iM_jD_{ij}}
	&= \frac{d_i}{c_{\rm tot}} 
	= \na x_i + \frac{F}{RT}\bigg(z_ix_i - (z\cdot x)\frac{\rho_i}{\rho_{\rm tot}}
	\bigg)\na\Phi, \label{2.Ji}
\end{align}
and the relations
$$
  c_i = \frac{\rho_i}{M_i}, \quad x_i = \frac{\rho_i}{c_{\rm tot}M_i}, \quad
	c_{\rm tot} = \sum_{i=1}^n c_i.
$$
Equations \eqref{1.rhoi}-\eqref{1.phi} 
are obtained from \eqref{2.rhoi}-\eqref{2.Ji} after
setting $\lambda=\eps_0(1+\chi)/F$, $k_{ij}=1/(c_{\rm tot}^3 M_iM_jD_{ij})$,
and $D_i=d_i/c_{\rm tot}$ and after nondimensionalization.
In particular, we scale the particle densities by $\rho_{\rm tot}$
(then the scaled quantities satisfy $\sum_{i=1}^n\rho_i=1$) and 
the electric potential by $F/(RT)$. 


\section{Auxiliary results}\label{sec.aux}

We collect some auxiliary results needed for the existence analysis.
The starting point is the relation \eqref{1.Ji} below.
Observe that the coefficients $k_{ij}$ depend on $\rho_i$ via $c_{\rm tot}
=\sum_{i=1}^n\rho_i/M_i$. This dependency does not complicates the analysis
since the results in Section \ref{sec.aux} hold pointwise for any given $\rho_i$
and $c_{\rm tot}$ is uniformly bounded from above and below by
$$
  \frac{1}{\max_{i=1,\ldots,n}M_i} \le c_{\rm tot} 
	= \sum_{i=1}^n\frac{\rho_i}{M_i} \le \frac{1}{\min_{i=1,\ldots,n}M_i}.
$$

\subsection{Expressions for the diffusion fluxes}\label{sec.flux}

We review three different expressions for the diffusion fluxes following
\cite{ChJu15,JuSt13} and extend the formulas to electro-chemical potentials. 
We reformulate \eqref{1.Ji}:
\begin{equation}\label{3.Di}
  D_i = -\sum_{j\neq i}k_{ij}(\rho_jJ_i-\rho_iJ_j)
	= \sum_{j\neq i}k_{ij}\rho_i\rho_j\bigg(\frac{J_i}{\rho_i}-\frac{J_j}{\rho_j}\bigg).
\end{equation}
The symmetry of $(k_{ij})$ implies that $\sum_{i=1}^n D_i=0$.
Compactly, we may write $D=-AJ$, where $D=(D_1,\ldots,D_n)^\top$,
$J=(J_1,\ldots,J_n)^\top$, and $A=(A_{ij})$ with
\begin{equation}\label{2.A}
  A_{ij} = \left\{\begin{array}{ll}
	\sum_{\ell=1,\,\ell\neq i}^n k_{i\ell}\rho_\ell &\quad\mbox{for }i=j, \\
	-k_{ij}\rho_i &\quad\mbox{for }i\neq j.
	\end{array}\right.
\end{equation}
By Assumption (A3), it holds that $\text{im}(A)=\text{ker}(A^\top)^\perp
=\text{span}\{\bm{1}\}^\perp$, where $\bm{1}=(1,\ldots,1)^\top$ $\in\R^n$. 
We conclude from \cite[Lemma 2.2]{JuSt13} that all eigenvalues of
$\widetilde A:=A|_{\text{im}(A)}$ are positive uniformly in $\rho\in[0,1]^n$
and that $\widetilde A$ is invertible.
Since $\sum_{i=1}^nJ_i=0$, each row of $J=(J_1,\ldots,J_n)$ is an element
of $\text{im}(A)$, so the linear system $D=-\widetilde AJ$ can be inverted,
yielding $J=-\widetilde A^{-1}D$. 

We obtain another formulation by inverting the system in the first $n-1$ variables.
Setting $D'=(D_1,\ldots,D_{n-1})$ and $J'=(J_1,\ldots,J_{n-1})$, we can write
$D'=-A_0J'$, where the matrix $A_0=(A_{ij}^0)\in\R^{(n-1)\times(n-1)}$ is
defined by
$$
  A_{ij}^0 = \left\{\begin{array}{ll}
	\sum_{\ell=1,\,\ell\neq i}^{n-1}(k_{i\ell}-k_{in})\rho_\ell + k_{in}
	&\quad\mbox{if }i=j, \\
	-(k_{ij}-k_{in})\rho_i &\quad\mbox{if }i\neq j.
	\end{array}\right.
$$
It is shown in \cite[Lemma 4]{ChJu15} that $A_0$ is invertible and $A_0^{-1}$
is bounded uniformly in $\rho\in[0,1]^n$. Thus, $J'=-A_0^{-1}D'$.

Finally, we invert the relations \eqref{3.Di}. Using $J_n=-\sum_{i=1}^{n-1}J_i$,
these relations (or the equivalent form $D_i=-\sum_{j=1}^n A_{ij}J_j$) 
can be written as
\begin{equation}\label{3.DCJ}
  \frac{D_i}{\rho_i} - \frac{D_n}{\rho_n} = -\sum_{j=1}^{n-1} C_{ij}J_j,
\end{equation}
where
\begin{align*}
  C_{ij} &= \frac{A_{ij}}{\rho_i} - \frac{A_{in}}{\rho_i}
	- \frac{A_{nj}}{\rho_n} + \frac{A_{nn}}{\rho_n}
	= -\frac{Y_{ij}}{\rho_i\rho_j} + \frac{Y_{in}}{\rho_i\rho_n}
	+ \frac{Y_{nj}}{\rho_n\rho_j} - \frac{Y_{nn}}{\rho_n^2}, \\
  Y_{ij} &= \left\{\begin{array}{ll}
	\sum_{\ell=1,\,\ell\neq i}^n k_{i\ell}\rho_i\rho_\ell &\quad\mbox{for }i=j, \\
	-k_{ij}\rho_i\rho_j &\quad\mbox{for }i\neq j.
	\end{array}\right.
\end{align*}
The matrix $-Y=(-Y_{ij})\in\R^{n\times n}$ is symmetric (since $(k_{ij})$ is 
symmetric), quasi-positive, irreducible, and it has the strictly positive eigenvector
$\bm{1}$ with eigenvalue zero. Hence, by the Perron-Frobenius theorem,
the spectral bound of $(-Y_{ij})$ is a simple eigenvalue (with value zero) 
and the spectrum of $(Y_{ij})$ consists of numbers with positive real part
and zero. Thus, $Y$ is positive semidefinite. 

We claim that the matrix $C=(C_{ij})\in\R^{(n-1)\times(n-1)}$ is positive definite
on $\operatorname{span}\{\bm{1}\}^\perp$.
Indeed, let $y\in\operatorname{span}\{\rho\}^\perp$. Then $y\cdot\rho=0$.
Since $\bm{1}\cdot\rho=1$, we have $y\not\in\operatorname{span}\{\bm{1}\}
=\operatorname{ker}(Y)$ and consequently, $\operatorname{span}\{\rho\}^\perp
\subset\operatorname{ker}(Y)^c$. This means that $-Y$ is negative definite
on $\operatorname{span}\{\rho\}^\perp$. A computation shows that for any
vector $w=(w_1,\ldots,w_{n-1})\in\R^{n-1}$, it holds that
$$
  \sum_{i,j=1}^{n-1}C_{ij}w_iw_j 
	= -\sum_{i,j=1}^{n}\frac{Y_{ij}}{\rho_i\rho_j}\widetilde w_i\widetilde w_j
$$
where $\widetilde w_i=w_i$ for $i=1,\ldots,n-1$ and
$\widetilde w_n=-\sum_{i=1}^{n-1}w_i$. Then $\widetilde w=(\widetilde w_1,\ldots,
\widetilde w_n)\in\operatorname{span}\{\bm{1}\}^\perp$. Since
$-Y$ is negative definite on $\operatorname{span}\{\rho\}^\perp$,
we infer that $(-Y_{ij}/(\rho_i\rho_j))$ 
is negative definite on $\operatorname{span}\{\bm{1}\}^\perp$.
Therefore, $C$ is positive definite on $\operatorname{span}\{\bm{1}\}^\perp$.
Its inverse $B:=c_{\rm tot}C^{-1}$ with $B=(B_{ij})$ exists, only depends on the mass
density vector $\rho$, and is positive definite uniformly for all $\rho\in[0,1]^n$
satisfying $\sum_{i=1}^n\rho_i=1$ \cite[Lemma 10]{ChJu15}.
We deduce from \eqref{3.DCJ} and \eqref{1.Ji} that
\begin{align}
  J_i &= -\sum_{j=1}^{n-1}B_{ij}\bigg(
	\frac{D_j}{\rho_j}-\frac{D_n}{\rho_n}\bigg) \nonumber \\
	&= -\sum_{j=1}^{n-1}B_{ij}\bigg(\frac{\na\log x_j}{M_j} - \frac{\na\log x_n}{M_n}
	+ \bigg(\frac{z_j}{M_j}-\frac{z_n}{M_n}\bigg)\na\Phi\bigg) \nonumber \\
	&= -\sum_{j=1}^{n-1}B_{ij}\na w_j	\label{3.Ji}
\end{align}
for $i=1,\ldots,n-1$ and $J_n=-\sum_{i=1}^{n-1}J_i$, recalling definition
\eqref{1.wi} of $w_i$. We summarize:

\begin{lemma}[Formulations of $J_i$]\label{lem.Ji}
Equations \eqref{3.Di} can be written equivalently as
$$
  J = -\widetilde A^{-1}D, \quad J' = -A_0^{-1}D', \quad J' = -B\na w.
$$
\end{lemma}

The last expression for $J_i$ shows that the partial mass balances \eqref{1.rhoi}
can be formulated as
$$
  \pa_t\rho' - \diver(B\na w) = r'(\rho),
$$
where $\rho=\rho(w)$ and $B=B(\rho(w))$. By Definition \eqref{1.wi},
$w$ is a function of $\rho$ (and $\Phi$). 
The inverse relation $\rho(w)$ is discussed in the following subsection.

\subsection{Inversion of $\rho\mapsto w$}

Definition \eqref{1.wi} defines, for given $\Phi\in\R$, a mapping
$x\mapsto w$. We claim that this mapping can be inverted.
If the molar masses are all the same, $M:=M_i$, this can be done explicitly:
\begin{equation}\label{3.rhow}
  \rho_i(w) = \frac{\exp(Mw_i-(z_i-z_n)\Phi)}{1 + \sum_{j=1}^{n-1}
	\exp(Mw_j - (z_j-z_n)\Phi)}, \quad i=1,\ldots,n-1,
\end{equation}
and $\rho_n=1-\sum_{i=1}^{n-1}\rho_i$. Unfortunately, when the molar masses 
are different, we cannot derive an explicit formula. 
Instead we adapt first Lemma 6 in \cite{ChJu15}.

\begin{lemma}[Inversion of $w$ and $x$]\label{lem.inv.wx}
Let $\Phi\in\R$ and define the function
$$
  W_\Phi:\bigg\{x=(x_1,\ldots,x_n)\in(0,1)^{n}:\sum_{i=1}^{n}x_i=1\bigg\}
	\to \R^{n-1}
$$
by $W_\Phi(x)=(w_1(x),\ldots,w_{n-1}(x))$, where
$$
  w_i(x) = \frac{\log x_i}{M_i} - \frac{\log x_n}{M_n}
	+ \bigg(\frac{z_i}{M_i}-\frac{z_n}{M_n}\bigg)\Phi, \quad i=1,\ldots,n-1.
$$
Then $W_\Phi$ is invertible and we can define 
$x'(w,\Phi):=W^{-1}_\Phi(w)$ and 
$x_n(w,\Phi):=1-\sum_{i=1}^{n-1}x_i$, where $x'(w,\Phi)=(x_1,\ldots,x_{n-1})$.
\end{lemma}

\begin{proof}
The proof is similar to that one of \cite[Lemma 6]{ChJu15}. 
Let $w=(w_1,\ldots,w_{n-1})\in\R^{n-1}$ and $\Phi\in\R$ be given. 
Define the function $f:[0,1]\to[0,\infty)$ by
$$
  f(s) = \sum_{i=1}^{n-1}(1-s)^{M_i/M_n}\exp\bigg[
	M_iw_i-M_i\bigg(\frac{z_i}{M_i}-\frac{z_n}{M_n}\bigg)\Phi\bigg], \quad
	s\in[0,1].
$$
Then $f$ is continuous, strictly decreasing, and $0=f(1)<f(s)<f(0)$ for $s\in(0,1)$.
Hence, there exists a unique fixed point $s_0\in(0,1)$ such that $f(s_0)=s_0$.
We define 
\begin{equation}\label{3.aux1}
  x_i = (1-s_0)^{M_i/M_n}\exp\bigg[
	M_iw_i-M_i\bigg(\frac{z_i}{M_i}-\frac{z_n}{M_n}\bigg)\Phi\bigg] > 0, \quad
	i=1,\ldots,n-1.
\end{equation}
By definition, we have $\sum_{i=1}^{n-1}x_i=f(s_0)=s_0<1$. We set $x_n=1-s_0>0$
such that $\sum_{i=1}^n x_i=1$. Moreover, \eqref{3.aux1} can be written
equivalently as
$$
  \frac{\log x_i}{M_i} + \frac{\log(1-s_0)}{M_n}
	+ \bigg(\frac{z_i}{M_i}-\frac{z_n}{M_n}\bigg)\Phi = w_i,
$$
and since $1-s_0=x_n$, this shows that $W^{-1}_\Phi(w)=x'$ is the inverse mapping.
\end{proof}

Given $\rho\in[0,1]^n$, we know that $x_i=\rho_i/(c_{\rm tot}M_i)$ for
$i=1,\ldots,n$ and $\sum_{i=1}^n x_i=1$. This relation can be inverted too.
We recall \cite[Lemma 7]{ChJu15}:

\begin{lemma}[Inversion of $\rho$ and $x$]\label{lem.inv.rhox}
Let $x'\in(0,1)^{n-1}$ and $x_n=1-\sum_{i=1}^{n-1}x_i>0$ be given and define
for $i=1,\ldots,n$,
$$
  \rho_i(x') = \rho_i := c_{\rm tot}M_ix_i, \quad\mbox{where }
	c_{\rm tot} = \bigg(\sum_{j=1}^n M_jx_j\bigg)^{-1}.
$$
Then $\rho=(\rho_1,\ldots,\rho_n)$ is the unique vector satisfying
$\rho_n=1-\sum_{i=1}^{n-1}\rho_i>0$, $x_i=\rho_i/(c_{\rm tot}M_i)$ for $i=1,\ldots,n$,
and $c_{\rm tot}=\sum_{i=1}^n\rho_i/M_i$. 
\end{lemma}

Combining Lemmas \ref{lem.inv.wx} and \ref{lem.inv.rhox}, we conclude as in
\cite{ChJu15} that the mapping $\rho\mapsto w$ can be inverted.
In fact, we just have to define $\rho'=\rho'(x'(w,\Phi))$.

\begin{corollary}[Inversion of $\rho$ and $w$]\label{coro.inv}
Let $w=(w_1,\ldots,w_{n-1})\in\R^{n-1}$ and $\Phi\in\R$ be given. 
Then there exists a unique vector
$\rho=(\rho_1,\ldots,\rho_n)\in(0,1)^n$ satisfying $\sum_{i=1}^n\rho_i=1$
such that \eqref{1.wi} holds for $\rho_n=1-\sum_{i=1}^{n-1}\rho_i$ and
$x_i=\rho_i/(c_{\rm tot}M_i)$ with $c_{\rm tot}=\sum_{i=1}^n\rho_i/M_i$. 
The mapping $\rho':\R^{n-1}\to
(0,1)^{n-1}$, $\rho'(w,\Phi)=(\rho_1,\ldots,\rho_{n-1})$, is bounded. 
\end{corollary}


\section{Proof of Theorem \ref{thm.ex}}\label{sec.ex}

{\em Step 1: existence of solutions.}
The idea is to apply the Leray-Schauder fixed-point theorem. We need to define
the fixed-point operator. For this, let $\chi\in L^\infty(\Omega;\R^{n-1})$ 
and $\sigma\in[0,1]$.
There exists a unique solution $\Phi^k-\Phi_D\in P_N$ to the linear finite-dimensional
problem
$$
  \lambda\int_\Omega\na\Phi^k\cdot\na\theta dy = \int_\Omega\bigg(\sum_{i=1}^n
	z_ic_i(\chi+w_D,\Phi^k) + f(y)\bigg)\theta dy
$$
for all $\theta\in P_N$. In particular, $\Phi^k\in L^\infty(\Omega)$. 
Next, we wish to solve the linear finite-dimensional problem
\begin{equation}\label{4.LM}
  a(u,\phi) = \sigma F(\phi)\quad\mbox{for all }\phi\in V_N,
\end{equation}
where
\begin{align*}
  a(u,\phi) &= \int_\Omega\na\phi:B(\chi+w_D,\Phi^k)\na u dy 
	+ \eps\int_\Omega u\cdot\phi dy, \\
  F(\phi) &= -\frac{1}{\tau}\int_\Omega\big(\rho'(\chi+w_D,\Phi^k)
	- \rho'(u^{k-1}+w_D,\Phi^{k-1})\big)dy \\
	&\phantom{xx}{}+ \int_\Omega r'(x(\chi+w_D,\Phi^k))\cdot\phi dy
	- \int_\Omega\na\phi:B(\chi+w_D,\Phi^k)\na w_D dy
\end{align*}
for $u$, $\phi\in V_N$. Since $\chi+w_D\in L^\infty(\Omega;\R^{n-1})$ and
$\Phi^k\in L^\infty(\Omega)$, 
Corollary \ref{coro.inv} shows that $\rho(\chi+w_D,\Phi^k)$ is bounded.
We know from Section \ref{sec.flux} that the matrix
$B=B(\chi+w_D,\Phi^k)$ is positive definite and its elements
are bounded. We deduce that the forms $a$ and $F$ are continuous on $V_N$.
Exploiting the equivalence of the norms in the finite-dimensional space $V_N$,
we find that
$$
  a(u,u) \ge \eps\|u\|_{L^2(\Omega)}^2 \ge \eps K_N\|u\|_{H^1(\Omega)}^2
$$
for some constant $K_N>0$,
which implies that $a$ is coercive on $V_N$. By the Lax--Milgram lemma, there
exists a unique solution $u\in V_N\subset L^\infty(\Omega;\R^{n-1})$ to
\eqref{4.LM} satisfying
\begin{equation}\label{4.aux1}
  \eps KN\|u\|_{L^\infty(\Omega)}^2 \le a(u,u) = \sigma F(u) 
	\le K_F\|u\|_{H^1(\Omega)},
\end{equation}
and the constants $K_N$ and $K_F$ are independent of $\tau$ and $\sigma$.
This defines the fixed-point operator $S:L^\infty(\Omega;\R^{n-1})\times[0,1]
\to L^\infty(\Omega;\R^{n-1})$, $S(\chi,\sigma)=u$. Standard arguments show that
$S$ is continuous. Since $V_N$ is finite-dimensional, $S$ is also compact.
Furthermore, $S(\chi,0)=0$. Estimate \eqref{4.aux1} provides a uniform bound
for all fixed points of $S(\cdot,\sigma)$. Thus, by the Leray-Schauder
fixed-point theorem, there exists $u^k\in V_N$ such that $S(u^k,1)=u^k$,
and $w^k:=u^k+w_D$, $\Phi^k$ solve \eqref{1.sch1}-\eqref{1.sch2}.

{\em Step 2: proof of the discrete entropy production inequality \eqref{1.ei}.} 
We use the test function
$\tau(w^k-w_D)\in V_N$ in \eqref{1.sch1} and set $\rho^k:=\rho'(w^k,\Phi^k)$:
\begin{align*}
  \int_\Omega&(\rho^k-\rho^{k-1})\cdot(w^k-w_D) dy
	+ \tau\int_\Omega\na(w^k-w_D):B(w^k,\Phi^k)\na w^k dy \\
  & {}+\eps\tau\int_\Omega|w^k-w_D|^2 dy 
	\le \tau\int_\Omega r'(x^k)\cdot(w^k-w_D)dy.
\end{align*}
We claim that the first term on the left-hand side is the difference of
the entropies at time steps $k$ and $k-1$. To show this, we split the entropy
density into two parts, $h(\rho^k)=h_1(\rho^k)+h_2(\rho^k)$, where
$$
  h_1(\rho^k) = c^k_{\rm tot}\sum_{i=1}^n x_i^k\log x_i^k, \quad
	h_2(\Phi^k) = \frac{\lambda}{2}|\na(\Phi^k-\Phi_D)|^2,
$$
where we recall that $x_i^k=\rho_i^k/(c_{\rm tot}^kM_i)$ and 
$c_{\rm tot}^k=\sum_{i=1}^n\rho_i^k/M_i$. By the convexity of $h_1$, we have
$$
  h_1(\rho^k)-h_1(\rho^{k-1})\le \frac{\pa h_1}{\pa\rho'}(\rho^k)\cdot(\rho^k-\rho^{k-1})
	= \sum_{i=1}^n(\rho^k_i-\rho_i^{k-1})\frac{\log x_i^k}{M_i}.
$$
Therefore, using $\rho_n^k-\rho_n^{k-1}=-\sum_{i=1}^{n-1}(\rho_i^k-\rho_i^{k-1})$,
\begin{align}
  \int_\Omega\big(h_1(\rho^k)-h_1(\rho^{k-1})\big)dx
	&\le \int_\Omega\bigg(\sum_{i=1}^{n-1}(\rho_i^k-\rho_i^{k-1})\frac{\log x_i^k}{M_i}
	+ (\rho_n^k-\rho_n^{k-1})\frac{\log x_n^k}{M_n}\bigg)dy \nonumber \\
	&= \int_\Omega\sum_{i=1}^{n-1}(\rho_i^k-\rho_i^{k-1})
	\bigg(\frac{\log x_i^k}{M_i}-\frac{\log x_n^k}{M_n}\bigg)dy. \label{4.aux2}
\end{align}

For the estimate of $h_2$, we first observe that
\begin{align*}
  \sum_{i=1}^{n-1}(\rho_i^k-\rho_i^{k-1})\bigg(\frac{z_i}{M_i}-\frac{z_n}{M_n}\bigg)
	&= \sum_{i=1}^{n-1}(\rho_i^k-\rho_i^{k-1})\frac{z_i}{M_i}
	+ (\rho_n^k-\rho_n^{k-1})\frac{z_n}{M_n} \\
	&= \sum_{n=1}^n(\rho_i^k-\rho_i^{k-1})\frac{z_i}{M_i}.
\end{align*}
We infer from the Poisson equation \eqref{1.sch2} and Young's inequality that
\begin{align}
  \int_\Omega&\sum_{i=1}^{n-1}(\rho_i^k-\rho_i^{k-1})
	\bigg(\frac{z_i}{M_i}-\frac{z_n}{M_n}\bigg)(\Phi^k-\Phi_D)dy \nonumber \\
	&= \int_\Omega\sum_{i=1}^{n}(\rho_i^k-\rho_i^{k-1})\frac{z_i}{M_i}(\Phi^k-\Phi_D)dy 
	= \int_\Omega\sum_{i=1}^n z_i(c_i^k-c_i^{k-1})(\Phi^k-\Phi_D)dy \nonumber \\
	&= \lambda\int_\Omega\na\big((\Phi^k-\Phi_D)-(\Phi^{k-1}-\Phi_D)\big)
	(\Phi^k-\Phi_D)dy \nonumber \\
	&\ge \frac{\lambda}{2}\int_\Omega|\na(\Phi^k-\Phi_D)|^2 dy
	- \frac{\lambda}{2}\int_\Omega|\na(\Phi^{k-1}-\Phi_D)|^2 dy \nonumber \\
	&= \int_\Omega\big(h_2(\Phi^k)-h_2(\Phi^{k-1})\big)dy. \label{4.aux3}
\end{align}
Taking into account the property $r_n(\rho^k)=-\sum_{i=1}^{n-1}r_i(\rho^k)$,
definition \eqref{1.wi} of $w_i^k$, and Assumption (A4), we compute
\begin{align}
  \int_\Omega & r'(x^k)\cdot(w^k-w_D) dy
	= \int_\Omega\sum_{i=1}^{n-1}r_i(x^k)\bigg(\frac{\log x_i^k}{M_i}
	- \frac{\log x_n^{k}}{M_n}\bigg)dy \nonumber \\
	&\phantom{xx}{}+ \int_\Omega\sum_{i=1}^{n-1}r_i(x^k)\bigg(\frac{z_i}{M_i}
	- \frac{z_n}{M_n}\bigg)(\Phi^k-\Phi_D)dy \nonumber \\
	&= \int_\Omega\sum_{i=1}^n r_i(x^k)\frac{\log x_i^k}{M_i}dy
	+ \int_\Omega\sum_{i=1}^n r_i(x^k)\frac{z_i}{M_i}(\Phi^k-\Phi_D)dy \nonumber \\
	&\le \int_\Omega\sum_{i=1}^n r_i(x^k)\frac{z_i}{M_i}(\Phi^k-\Phi_D)dy,
	\label{4.aux4}
\end{align}
Combining \eqref{4.aux2}-\eqref{4.aux4} gives the conclusion.


\section{Proof of Theorem \ref{thm.conv}}\label{sec.conv}

Let $(w^k,\Phi^k)$ be a weak solution to scheme \eqref{1.sch1}-\eqref{1.sch2}
and define $\rho^k=\rho(w^k,\Phi^k)$.

{\em Step 1: uniform estimates.}
We derive estimates for $\rho^k$ and $\Phi^k$ independent of
$\eps$, $\tau$, and $N$. The starting point is the discrete entropy production
inequality \eqref{1.ei}, and the main task is to estimate the diffusion part.

\begin{lemma}[Estimate of the diffusion part]\label{lem.diff}
There exist constants $K_1>0$ and $K_2>0$, both independent of $\eps$, $\tau$, and $N$,
such that
$$
  \int_\Omega\na(w^k-w_D):B\na w^k dy
	\ge K_1\sum_{i=1}^n\|\na (x_i^k)^{1/2}\|_{L^2(\Omega)}^2 - K_2.
$$
\end{lemma}

\begin{proof}
We drop the superindex $k$ in the proof to simplify the notation.
Recall that $\widetilde A=A|_{\text{im}(A)}$, where 
$\mbox{im}(A)=\operatorname{span}\{\bm{1}\}^\perp$. We introduce as in the proof
of Lemma 12 in \cite{ChJu15} the symmetrization 
$\widetilde A_S=P^{-1/2}\widetilde AP^{1/2}$, where
$P^{1/2}=M^{1/2}X^{1/2}$ and
$M^{1/2}:=\operatorname{diag}(M_1^{1/2},$ $\ldots,M_n^{1/2})$, 
$X^{1/2}:=\operatorname{diag}(x_1^{1/2},\ldots,x_n^{1/2})$.
Then $\widetilde A_S^{-1}=P^{-1/2}\widetilde A^{-1}P^{1/2}$ 
is a self-adjoint endomorphism
whose smallest eigenvalue is bounded from below by some positive constant which
depends only on $(k_{ij})$. 

Since $0=\sum_{i=1}^n J_i=\sum_{i=1}^n (B\na w)_i$, we can express the last component
in terms of the other components, $(B\na w)_n = -\sum_{i=1}^{n-1}(B\na w)_i$.
Then
\begin{align*}
  \na w:B\na w &= \sum_{i=1}^{n-1}\bigg\{\frac{\na\log x_i}{M_i}
	-\frac{\na\log x_n}{M_n} + \bigg(\frac{z_i}{M_i}-\frac{z_n}{M_n}\bigg)
	\na\Phi\bigg\}\cdot(B\na w)_i \\
	&= \sum_{i=1}^{n-1}\frac{1}{M_i}\na(\log x_i+z_i\Phi)\cdot(B\na w)_i
	- \frac{1}{M_n}\na(\log x_n+z_n\Phi)\sum_{i=1}^{n-1}(B\na w)_i \\
  &= \sum_{i=1}^n\frac{1}{M_i}\na(\log x_i+z_i\Phi)\cdot(B\na w)_i.
\end{align*}
To simplify the notation, we set $\Psi_i=\na(\log x_i+z_i\Phi)/M_i$,
and $\Psi=(\Psi_1,\ldots,\Psi_n)$.
By Lemma \ref{lem.Ji}, $B\na w=\widetilde A^{-1}D 
= P^{1/2}\widetilde A_S^{-1}P^{-1/2}D$. Hence,
\begin{align}
  \na w:B\na w &= \Psi:B\na w = \Psi:M^{1/2}X^{1/2}\widetilde A_S^{-1}
	X^{-1/2}M^{-1/2}D \nonumber \\
	&= \sum_{i,j=1}^n\Psi_i M_i^{1/2}x_i^{1/2}(\widetilde A_S^{-1})_{ij}
	x_j^{-1/2}M_j^{-1/2}D_i \nonumber \\
	&= \sum_{i,j=1}^n\big(2\na x_i^{1/2}+z_ix_i^{1/2}\na\Phi\big)M_i^{-1/2}
	(\widetilde A_S^{-1})_{ij}M_j^{-1/2} \nonumber \\
	&\phantom{xx}{}\times 
	\big(2\na x_j^{1/2} + (z_jx_j^{1/2} - (x\cdot z)\rho_jx_j^{-1/2})\na\Phi\big).
	\label{4.aux5}
\end{align}
In view of $\sum_{i=1}^{n}(B\na w)_i=0$, it follows that 
\begin{align*}
  \sum_{i,j=1}^n & \big(M_i^{-1/2} x_i^{-1/2}(z\cdot x)\rho_i\na\Phi\big)
	(\widetilde A_S)^{-1}_{ij}
	M_j^{-1/2}\big(2\na x_j^{1/2} + (z_jx_j^{1/2} - (x\cdot z)\rho_jx_j^{-1/2})
	\na\Phi\big) \\
	&= \sum_{i,j=1}^n\big(c(z\cdot x)\na\Phi\big)\widetilde A^{-1}_{ij}
	\big(\na x_j + (z_jx_j - (x\cdot z)\rho_j\na\Phi\big) \\
	&= \big(c(z\cdot x)\na\Phi\big)\cdot\sum_{i=1}^n(B\na w)_{i} = 0.
\end{align*}
Adding this expression to \eqref{4.aux5}, we find that
\begin{align*}
  \na w:B\na w &= \sum_{i,j=1}^n M_i^{-1/2}
	\big(2\na x_i^{1/2} + (z_ix_i^{1/2} - (z\cdot x)\rho_i x_i^{-1/2}\na\Phi\big)
	(\widetilde A_S)^{-1}_{ij}M_j^{-1/2} \\
	&\phantom{xx}{}\times\big(2\na x_j^{1/2} + (z_jx_j^{1/2} 
	- (z\cdot x)\rho_j x_j^{-1/2}\na\Phi\big).
\end{align*}
The matrix $\widetilde A_S^{-1}$ is positive definite on 
$\mbox{im}(\widetilde A_S)=\operatorname{span}\{\rho^{1/2}\}$.
As the vector $(2\na x_i^{1/2} + (z_ix_i^{1/2} - (x\cdot z)\rho_i x_i^{-1/2}
\na\Phi)_{i=1}^n$ lies in $\operatorname{span}\{\rho^{1/2}\}$, we obtain
\begin{align*}
  \na w:B\na w &\ge K_B\sum_{i=1}^n M_i^{-1}\big|2\na x_i^{1/2}
	+ (z_ix_i^{1/2} - (x\cdot z)\rho_i x_i^{-1/2}\na\Phi\big|^2 \\
	&\ge K_1\sum_{i=1}^n|\na x_i^{1/2}|^2
	- K_2\sum_{i=1}^n \big|(z_ix_i^{1/2} - (x\cdot z)\rho_i x_i^{-1/2}\na\Phi\big|^2,
\end{align*}
where $K_1>0$ and $K_2>0$ depend on $M_1,\ldots,M_n$. Since $x_i$ and
$\rho_i x_i^{-1/2}=\rho_i^{1/2}/(c_{\rm tot}M_i)$ are bounded, the previous
inequality becomes
\begin{equation}\label{4.part1}
  \na w:B\na w \ge K_1\sum_{i=1}^n|\na x_i^{1/2}|^2 - K_3|\na\Phi|^2,
\end{equation}
where $K_3$ depends on $K_2$ and $z_i$.

In the following, let $K>0$ be a generic constant independent of $\eps$, $n$,
and $\tau$. We estimate the expression involving the boundary term
\begin{align*}
  \na w_D:B\na w &= \na w_D:A_0^{-1}D' \\
	&= \sum_{i,j=1}^{n-1} (A_0^{-1})_{ij}
	\bigg(\frac{z_i}{M_i}-\frac{z_n}{M_n}\bigg)\na\Phi_D\cdot
  \big(\na x_i + (z_ix_i - (z\cdot x)\rho_i)\na\Phi\big) \\
	&\le \frac{K}{\delta}
	+ \delta\sum_{i=1}^{n-1}\big|\na x_i + (z_ix_i - (z\cdot x)\rho_i)\na\Phi\big|^2,
\end{align*}
where $K>0$ depends on $\na\Phi_D$, $z_i$, $M_i$, and $A_0^{-1}$.
Since $0\le x_i\le 1$, we have $|\na x_i|^2=4x_i|\na x_i^{1/2}|^2\le4|\na x_i^{1/2}|^2$
and therefore, 
\begin{equation}\label{4.part2}
  \na w_D:B\na w \le \frac{K}{\delta} + 4\delta|\na x_i^{1/2}|^2
	+ \delta K|\na\Phi|^2.
\end{equation}
We infer from \eqref{4.part1} and \eqref{4.part2} that
$$
  \int_\Omega\na(w-w_D):B\na w dy
	\ge (K_1-4\delta)\sum_{i=1}^n\|\na x_i^{1/2}\|_{L^2(\Omega)}^2
  - K_3\|\na\Phi\|^2_{L^2(\Omega)} - \frac{K}{\delta}.
$$
By the boundedness of $c_i$, the elliptic estimate for the Poisson equation gives
\begin{equation}\label{4.Phi}
  \|\Phi\|_{H^1(\Omega)} \le K(1 + \|c_i\|_{L^2(\Omega)}) \le K.
\end{equation}
This proves the lemma.
\end{proof}

Combining the discrete entropy inequality \eqref{1.ei} and the estimate of
Lemma \ref{lem.diff} and summation over $k$ leads to the following result.

\begin{corollary}\label{coro.ei}
There exist constants $K_1>0$ and $K_2>0$, both independent of $\eps$, $n$,
and $\tau$, such that
\begin{equation}\label{4.Hk}
  H(\rho^k) + \tau K_1\sum_{j=1}^k\sum_{i=1}^n\|\na (x_i^k)^{1/2}\|_{L^2(\Omega)}^2
	+ \eps\tau\sum_{j=1}^k\|w^j-w_D\|_{L^2(\Omega)}^2
  \le \tau kK_2 + H(\rho^0).
\end{equation}
\end{corollary}

{\em Step 2: limit $\eps\to 0$.}
For a fixed time step $k$, let $(w^\eps,\Phi^\eps)$ 
be a solution to \eqref{1.sch1}-\eqref{1.sch2} with $\rho^\eps=\rho(w^\eps,\Phi^\eps)$ 
and $x^\eps_i=\rho_i^\eps/(c_{\rm tot}^\eps M_i)$.
Estimates \eqref{4.Phi} and \eqref{4.Hk} yield the following uniform bounds:
\begin{align}
  \|\rho_i^\eps\|_{L^\infty(\Omega)} + \|x_i^\eps\|_{L^\infty(\Omega)} &\le 1, 
	\quad i=1,\ldots,n, \label{4.Linfty} \\
  \|x_i^\eps\|_{H^1(\Omega)} + \|\Phi^\eps\|_{H^1(\Omega)}
	+ \eps^{1/2}\|w_i^\eps\|_{L^2(\Omega)} \label{4.H1}
	&\le K,
\end{align}
where $K>0$ is independent of $\eps$ and $N$.
The bound for $x_i^\eps$ in $H^1(\Omega)$ is a consequence of the bound 
for $(x_i^\eps)^{1/2}$ in $H^1(\Omega)$ from \eqref{4.Hk} and the
uniform $L^\infty$ bound for $x_i^\eps$ from \eqref{4.Linfty}.
It follows that $c_{\rm tot}^\eps = \sum_{i=1}^n \rho_i^\eps/M_i$ is uniformly
bounded in $L^\infty(\Omega)$. Moreover, because of $\sum_{i=1}^n\rho_i^\eps=1$, 
$c_{\rm tot}^\eps\ge (\max_i M_i)^{-1}>0$
is uniformly positive. This shows that $\rho_i^\eps=c_{\rm tot}^\eps M_ix_i^\eps$
is uniformly bounded in $H^1(\Omega)$. Oberserving that the embedding
$H^1(\Omega)\hookrightarrow L^2(\Omega)$ is compact, 
there exist subsequences, which are not relabeled, such that as $\eps\to 0$,
\begin{align*}
  x_i^\eps\to x_i, \quad \rho_i^\eps\to\rho_i, \quad \Phi^\eps\to\Phi
	&\quad\mbox{strongly in }L^2(\Omega), \\
	x_i^\eps\rightharpoonup x_i, \quad \rho_i^\eps\rightharpoonup\rho_i,
	\quad \Phi^\eps\rightharpoonup\Phi &\quad\mbox{weakly in }H^1(\Omega), \\
	\eps w_i^\eps\to 0 &\quad\mbox{strongly in }L^2(\Omega).
\end{align*}
In view of the $L^\infty$ bounds for $(x_i^\eps)$ and $(\rho_i^\eps)$,
the strong convergences for these (sub-) sequences hold in $L^p(\Omega)$ for any
$p<\infty$. Consequently, $c_{\rm tot}^\eps\to c_{\rm tot}:=\sum_{i=1}^n\rho_i/M_i$ 
strongly in $L^2(\Omega)$, and
we can identify $\rho_i=c_{\rm tot}M_ix_i$ for $i=1,\ldots,n$. Furthermore, 
$$
  c_i^\eps=\rho_i^\eps/M_i \to c_i:= \rho_i/M_i \quad\mbox{strongly in }
	L^2(\Omega),\ i=1,\ldots,n.
$$
Recalling definition \eqref{1.Ji} of $D_i$, we have
\begin{equation}\label{4.conv.D}
  D^\eps_i = \na x_i^\eps + (z_ix_i^\eps - (z\cdot x^\eps)\rho_i^\eps)\na\Phi^\eps
	\rightharpoonup D_i := \na x_i + (z_ix_i - (z\cdot x)\rho_i)\na\Phi
\end{equation}
weakly in $L^q(\Omega)$ for any $q<2$ and $i=1,\ldots,n$. 
Since $(D_i^\eps)$ is bounded in $L^2(\Omega)$, there exists a subsequence
which converges to some function $\widetilde D_i$ weakly in $L^2(\Omega)$. 
By the uniqueness of the weak limits, we can identify $\widetilde D_i=D_i$. 
This shows that the convergence \eqref{4.conv.D} holds in $L^2(\Omega)$.
We deduce from the strong convergence of
$(x_i^\eps)$, the boundedness of $(x_i^\eps)$ in $L^\infty(\Omega)$, and
the continuity of $r_i$ that $r_i(x^\eps)\to r_i(x)$ strongly in $L^2(\Omega)$.

We know from Lemma \ref{lem.Ji} that 
$B(w^\eps)\na w^\eps = A_0^{-1}(\rho^\eps)(D^\eps)'$. As $A_0^{-1}(\rho)$ is uniformly
bounded for $\rho\in[0,1]^n$ and $(\rho^\eps)$ converges strongly to $\rho$, we
infer that $A_0^{-1}(\rho^\eps)\to A_0^{-1}(\rho)$ strongly in $L^2(\Omega)$;
the convergence holds even in every $L^p(\Omega)$ for $p<\infty$. Then,
because of \eqref{4.conv.D},
\begin{equation}\label{4.A0D}
  A_0^{-1}(\rho^\eps)(D^\eps)'\rightharpoonup A_0^{-1}(\rho)D'
	\quad\mbox{weakly in }L^q(\Omega)\mbox{ for all }q<2.
\end{equation}
In fact, since $A_0^{-1}(\rho^\eps)(D^\eps)')$ is bounded in $L^2(\Omega)$
and thus (up to a subsequence) weakly converging in $L^2(\Omega)$, the
convergence holds in $L^2(\Omega)$.

These convergences are sufficient to perform the limit $\eps\to 0$ in
\eqref{1.sch1}-\eqref{1.sch2}. We conclude that $(\rho^k,\Phi^k):=(\rho,\Phi)$ solves
\begin{align}
  & \frac{1}{\tau}\int_\Omega\big((\rho^k)'-(\rho^{k-1})'\big)\cdot\phi dy
	+ \int_\Omega\na\phi:A_0^{-1}(\rho^k)\na\rho^k dy
	= \int_\Omega r'(x^k)\cdot \phi dy, \label{4.gal3} \\
	& \lambda\int_\Omega\na\Phi^k\cdot\na\theta dy 
	= \int_\Omega\bigg(\sum_{i=1}^n z_ic_i^k + f(y)\bigg)\theta dy \label{4.gal4}
\end{align}
for all $\phi\in V_N$, $\theta\in P_N$.

{\em Step 3: limit $N\to \infty$.} Let $(\rho^N,\Phi^N)$ be a solution to
\eqref{4.gal3}-\eqref{4.gal4}. Estimates \eqref{4.Linfty}-\eqref{4.H1} are independent
of $N$. Thus, we can exactly argue as in step 2 and obtain
limit functions $(x,\rho,\Phi)$ and $c_i=c_{\rm tot}M_ix_i$ for $i=1,\ldots,n$
as $N\to\infty$.
These functions satisfy \eqref{4.gal3}-\eqref{4.gal4} for all $\phi\in V_N$
and $\theta\in P_N$ and for all $N\in\N$. 
The union of all $V_N$ is dense in $H^1(\Omega;\R^{n-1})$ and the union of all
$P_N$ is dense in $H_D^1(\Omega)$. Thus, by a density argument, system
\eqref{4.gal3}-\eqref{4.gal4} holds for all test functions $\phi\in H^1(\Omega;\R^{n-1})$
and $\theta\in H_D^1(\Omega)$.

{\em Step 4: limit $\tau\to 0$.} Let $(\rho^k,\Phi^k)$ be a solution to
\eqref{4.gal3}-\eqref{4.gal4} with test functions $\phi\in H^1(\Omega;\R^{n-1})$
and $\theta\in H_D^1(\Omega)$. Then $\rho_i^k=c_{\rm tot}^kM_ix_i^k$
and $c_i^k=\rho^k_i/M_i$ for $i=1,\ldots,n$. We set
$$
  \rho^\tau_i(y,t) = \rho_i^k(y), \quad x_i^\tau(y,t) = x_i^k(y),
	\quad c_i^\tau(y,t) = c_i^k(y), \quad \Phi^\tau(y,t) = \Phi^k(y)
$$
for $y\in\Omega$, $t\in((k-1)\tau,k\tau]$, $i=1,\ldots,n$ and introduce
the shift operator $(\sigma_\tau\rho^\tau)(y,t)=\rho^\tau(y)$ for $y\in\Omega$
and $t\in((k-1)\tau,k\tau]$. Finally, we set
$D^\tau_i=\na x_i^\tau + (z_ix_i^\tau-(z\cdot x^\tau)\rho^\tau_i)\na\Phi^\tau$
and $T=m\tau$ for some fixed $m\in\N$.
Then we can write system \eqref{4.gal3}-\eqref{4.gal4} as
\begin{align}
  & \frac{1}{\tau}\int_0^T\int_\Omega\big((\rho^\tau)'-\sigma_\tau(\rho^\tau)'\big)
	\cdot\phi dydt + \int_0^T\int_\Omega\na\phi:A_0^{-1}(\rho^\tau)(D^\tau)' dydt 
	\nonumber \\
	&\phantom{xx}{}= \int_0^t\int_\Omega r'(x^\tau)\cdot\phi dydt, \label{4.tau1} \\
  & \lambda\int_\Omega\na\Phi^\tau\cdot\na\theta dy
	= \int_\Omega\bigg(\sum_{i=1}^n z_ic_i^\tau + f(y)\bigg)\theta dy \label{4.tau2}
\end{align}
for all piecewise constant functions $\phi:(0,T)\to H^1(\Omega;R^{n-1})$
and $\theta:(0,T)\to H_D^1(\Omega)$. The entropy inequality \eqref{4.Hk},
formulated in terms of $(\rho^\tau,\Phi^\tau)$, provides us with further uniform
bounds since the right-hand side of \eqref{4.Hk} does not depend on $\tau$:
\begin{align}
  \|\rho_i^\tau\|_{L^\infty(\Omega_T)} 
	+ \|x_i^\tau\|_{L^\infty(\Omega_T)} &\le K, \label{4.est1} \\
	\|\rho_i^\tau\|_{L^2(0,T;H^1(\Omega))} 
	+ \|x_i^\tau\|_{L^2(0,T;H^1(\Omega))} + \|\Phi^\tau\|_{L^2(0,T;H^1(\Omega))}
	&\le K, \label{4.est2}
\end{align}
where we have set $\Omega_T=\Omega\times(0,T)$.
As a consequence, $(D_i^\tau)$ is bounded in $L^2(0,T;H^1(\Omega))$. 

It remains to derive a uniform estimate for the discrete time derivative of
$\rho^\tau$. Taking into account the uniform bound for $A_0^{-1}(\rho^\tau)$,
it follows that
\begin{align*}
  \frac{1}{\tau}\bigg|\int_0^t\int_\Omega&\big((\rho^\tau)'-\sigma_\tau(\rho^\tau)'\big)
	\cdot\phi dydt\bigg|
	\le \int_0^T\|\na\phi\|_{L^2(\Omega)}\|A_0^{-1}(\rho^\tau)\|_{L^\infty(\Omega)}
	\|(D^\tau)'\|_{L^2(\Omega)}dt \\
	&\phantom{xx}{}+ \int_0^T\|r'(x^\tau)\|_{L^2(\Omega)}\|\phi\|_{L^2(\Omega)} dt
	\le C\|\phi\|_{L^2(0,T;H^1(\Omega))}.
\end{align*}
As the piecewise constant functions $\phi:(0,T)\to H^1(\Omega;\R^{n-1})$ are
dense in $L^2(0,T;$ $H^1(\Omega;\R^{n-1}))$, this estimate also holds for all
$\phi\in L^2(0,T;H^1(\Omega;\R^{n-1}))$, and we conclude that
$$
  \tau^{-1}\big\|(\rho^\tau)'-\sigma_\tau(\rho^\tau)'\big\|_{L^2(0,T;H^1(\Omega)')}
	\le K, \quad i=1,\ldots,n-1.
$$
This estimate also holds for $i=n$ since $\rho_n^\tau=1-\sum_{i=1}^{n-1}\rho_i^\tau$.

By the Aubin-Lions lemma in the version of \cite{DrJu12}, there exists a
subsequence of $(\rho^\tau)$ which is not relabeled such that, as $\tau\to 0$,
$$
  \rho_i^\tau\to \rho_i\quad\mbox{strongly in }L^2(\Omega_T),\ i=1,\ldots,n.
$$
In view of the $L^\infty$ bound \eqref{4.est1} for $\rho^\tau$, 
this convergence also holds in $L^p(\Omega_T)$ for any $p<\infty$. Furthermore, by
\eqref{4.est2}, we have up to subsequences,
\begin{align*}
  x^\tau_i\rightharpoonup x_i, \quad 
	\Phi^\tau\rightharpoonup\Phi &\quad\mbox{weakly in }L^2(0,T;H^1(\Omega)), \\
	\tau^{-1}(\rho^\tau_i-\sigma_\tau(\rho^\tau_i))\rightharpoonup\pa_t\rho_i
	&\quad\mbox{weakly in }L^2(0,T;H^1(\Omega)').
\end{align*}
In particular, $D^\tau_i\rightharpoonup D_i$ weakly in $L^2(\Omega_T)$,
and we can identify $D_i=\na x_i + (z_ix_i-(z\cdot x)\rho_i)\na\Phi$.
The strong convergence of $(\rho^\tau)$ and the weak convergence of $(D^\tau_i)$
imply that
$$
  A_0^{-1}(\rho^\tau)(D^\tau)'\rightharpoonup A_0^{-1}(\rho)D'
	\quad\mbox{weakly in }L^q(\Omega_T),\ q<2.
$$
Again, since $(A_0^{-1}(\rho^\tau)(D^\tau)')$ is bounded in $L^2(\Omega_T)$,
this convergence holds in $L^2(\Omega_T)$. Furthermore,
$r'(x^\tau)\to r'(x)$ strongly in $L^2(\Omega_T)$. Therefore,
we can pass to the limit $\tau\to 0$ in \eqref{4.tau1}-\eqref{4.tau2}
yielding \eqref{1.weak1}-\eqref{1.weak2}.

Finally, the assumption $\rho_i^0\ge\eta>0$ can be relaxed to $\rho_i^0\ge 0$
by passing to the limit $\eta\to 0$. This is carried out in \cite[Section 3.2]{ChJu15}
and we refer to this reference for details.


\section{Numerical experiments}\label{sec.numer}

In this section, some numerical experiments based on scheme 
\eqref{1.sch1}-\eqref{1.sch2} in one space dimension are presented.

\subsection{Discretization and iteration procedure}

Let $\Omega=(0,1)$ be divided into $n_p\in\N$ uniform subintervals of
length $h=1/n_p$. We use uniform time steps with time step size $\tau>0$
and linear finite elements. We impose Dirichlet boundary condition for
the electric potential $\Phi$. Given the variables $(w,\Phi)$, 
the molar fractions $x_i$ are computed from the
fixed-point problem (see the proof of Lemma \ref{lem.inv.wx})
\begin{equation}\label{4.fpp}
  f(s) = \sum_{i=1}^{n-1}(1-s)^{M_i/M_n}\exp\bigg[
	M_iw_i-M_i\bigg(\frac{z_i}{M_i}-\frac{z_n}{M_n}\bigg)\Phi_0\bigg], \quad
	s\in[0,1],
\end{equation}
with unique solution $s_0\in(0,1)$.
The molar fractions are recovered froms \eqref{3.aux1},
$$
  x_i = (1-s_0)^{M_i/M_n}\exp\bigg[
	M_iw_i-M_i\bigg(\frac{z_i}{M_i}-\frac{z_n}{M_n}\bigg)\Phi\bigg], \quad
	i=1,\ldots,n-1,
$$
and $x_n=1-s_0$. Then we set (see Lemma \ref{lem.inv.rhox})
$c_{\rm tot}=\sum_{i=1}^n (M_ix_i)^{-1}$ and $\rho_i=c_{\rm tot}M_ix_i$
for $i=1,\ldots,n$.

Instead of solving the nonlinear discrete system \eqref{1.sch1}-\eqref{1.sch2}
by a full Newton method, we employ a linearized semi-implicit approach,
i.e., we linearize $\rho(w,\Phi)$ and use the previous time step
in the diffusion matrix $B(w)$. More precisely, let $\overline{w}\in V_N$
and $\overline\Phi\in P_N$ be given. We linearize $\rho(w,\Phi)$ by 
$$
  \rho(\overline{w},\overline\Phi)+\na_{(w,\Phi)}\rho'(\overline{w},\overline\Phi)
  \cdot(w-\overline{w},\Phi-\overline\Phi).
$$
This leads to the problem in the variable $\zeta=(w-\overline{w},\Phi-\overline\Phi)$:
\begin{equation}\label{4.linsys}
  L(\zeta,\phi) = F(\phi), \quad K(\zeta_n,\theta) = G(\theta)
  \quad\mbox{for all }\phi\in V_N,\ \theta\in P_N,
\end{equation}
where
\begin{align*}
  L(\zeta,\phi) &= \int_\Omega\na_{(w,\Phi)}\rho'(\overline{w},\overline\Phi)\cdot
	(\zeta,\phi)dy + \tau\int_\Omega\pa_x\phi\cdot B(\overline{w},\overline{\Phi})
	\pa_x\zeta dy + \eps\tau\int_\Omega(\zeta-w_D)\cdot\phi dy, \\
	F(\phi) &= -\int_\Omega\big(\rho'(\overline{w},\overline\Phi)
	- \rho'(w^{k-1},\Phi^{k-1})\big)\cdot\phi dy
	- \tau\int_\Omega\pa_x\phi\cdot B(\overline{w},\overline\Phi)\pa_x \overline{w}dy, \\
  K(\zeta_n,\theta) &= \lambda\int_\Omega\pa_x\zeta_n\pa_x\phi dy
	- \int_\Omega\sum_{i=1}^n\frac{z_i}{M_i}
	\na_{(w,\Phi)}\rho_i(\overline{w},\overline\Phi)\cdot\zeta\theta dy, \\
	G(\theta) &= -\lambda\int_\Omega\pa_x\overline\Phi\pa_x\theta dy
	+ \int_\Omega\bigg(\sum_{i=1}^nz_i\frac{\rho_i(\overline{w},\overline\Phi)}{M_i}
	+f(y)\bigg)\theta dy.
\end{align*}
The iteration with starting point $(w_h^{(0)},\Phi_h^{(0)}):=(w^{k-1},\Phi^{k-1})$
is then defined by $(w_h^{(m+1)},$ $\Phi_h^{(m+1)}):=(\overline{w},\overline\Phi)
+\zeta$ for $m\ge 0$. The iteration stops when $\|\zeta\|_{\ell^\infty}<\eps_{\rm tol}$
for some tolerance $\eps_{\rm tol}>0$ or if $m\ge m_{\rm max}$ for a maximal number of
iterations. We summarize the scheme in Algorithm 1. 

\begin{algorithm}
\label{algo}
\begin{algorithmic}[1]
\Procedure{Maxwell-Stefan system in entropy variables}{}
\State \text{Set} 
$(\overline{w}_h^{(0)},\overline{\Phi}_h^{(0)}) = (w^{k-1},\Phi^{k-1})$, 
$\rho^{(0)}_h = \rho'(\overline{w}_h^{0},\overline{\Phi}_h^{0})$, 
$x^{(0)}_h = \rho_h^{(0)}/(M_i c_h^{(0)})$, 
$c_h^{(0)} = \sum_{i=1}^n (\rho_h^{(0)})_i/M_i$, $m=0$, $\eps>0$, and $m_{\text{max}}$.
\While{$err > \eps$}
	\State Solve linear system \eqref{4.linsys} with solution $\zeta$.
	\State Set $(\overline{w}^{(m+1)}_h,\overline{\Phi}^{(m+1)}_h)
	:= (\overline{w}_h^{m},\overline{\Phi}_h^{m}) + \zeta$.
	\State Solve the fixed-point problem \eqref{4.fpp}  with solution $s_0$.
	\State Compute $x^{(m+1)}_h$ and $\rho^{(m+1)}_h$.
	\State Set $\text{err}:= \|(\overline{w}_h^{(m+1)},\overline{\Phi}_h^{(m+1)})
	- (\overline{w}_h^{(m)},\overline{\Phi}_h^{(m)})\|_{\ell^{\infty}}$.
	\State $(m+1) \gets (m)$.
	\If{$m>m_{\text{max}}$ \mbox{or }$\text{err}<\eps$}
		\State \textbf{Break}
	\EndIf
\EndWhile
\EndProcedure
\end{algorithmic}
\caption{(Pseudo-code for the finite-element scheme in entropy variables.)}
\end{algorithm}

The linear system
\eqref{4.linsys} and the fixed-point problem \eqref{4.fpp} are solved using MATLAB.
We choose the numerical parameters $h=10^{-2}$, $\tau=10^{-3}$, 
$\eps_{\rm tol}=10^{-10}$, 
and $\eps=2^{-52}\approx 2.2204\cdot 10^{-16}$ (the scheme works also for
$\eps=0$). 

We have compared our results with the solutions from a finite-element scheme
derived from the original system in the variables $\rho_i$ and a Picard
iteration procedure for the nonlinear discrete system. It turned out that
the results are basically the same, i.e.\ $\|\rho_i-\rho_i(w,\Phi)\|_{L^\infty(\Omega)}
\le 10^{-10}$.


\subsection{Numerical examples}

In all numerical examples, we neglect reaction terms and choose the diffusivities
according to \cite{BGS12,GaMa92}:
$D_{12}=0.833$, $D_{13}=0.680$, and $D_{23}=0.168$ for $n=3$.
The charges are given by $z_1=z_2=1$ and $z_3=0$ and 
the initial data is defined as in \cite{BGS12}:
$$
  \rho_1^0(y) = \left\{\begin{array}{ll}
	0.7 & \mbox{for }y<0.25, \\
	-2(0.7-\eta)y - 2(0.25\eta - (0.7\cdot 0.75)) &\mbox{for }0.25\le y<0.75, \\
  \eta &\mbox{for }0.75\le y\le 1
  \end{array}\right.
$$
for $\eta=10^{-5}$, $\rho_2^0(y)=0.2$, and $\rho_3^0(y)=(1-\rho_1^0-\rho_2^0)(y)$
for $y\in\Omega=(0,1)$. 

For the first example, the boundary conditions
for the electric potential are supposed to be in equilibrium, i.e.\
$\Phi(y)=0$ for $y\in\{0,1\}$. The dynamics of the particle densities and the electric
potential are shown in Figure \ref{fig.ex1}. The solution at time $t=17$ is
essentially stationary and, in fact, in equilibrium. 
Because of the choice of the parameters,
the stationary solution is symmetric around $x=\frac12$. 

\begin{figure}[ht]
\includegraphics[width=75mm]{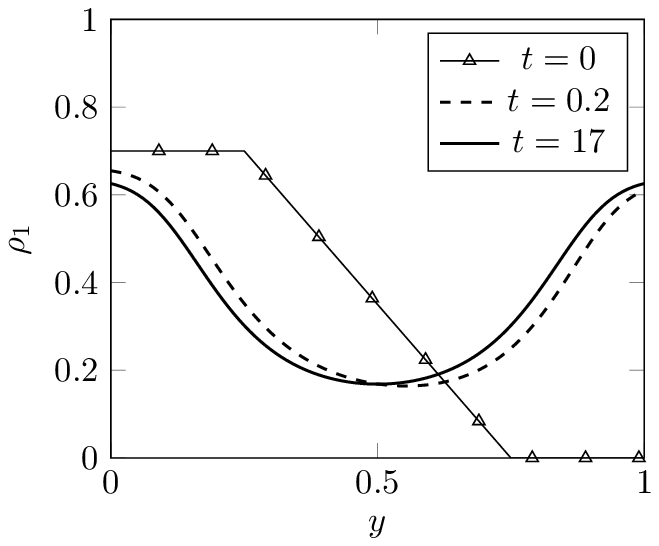}
\includegraphics[width=75mm]{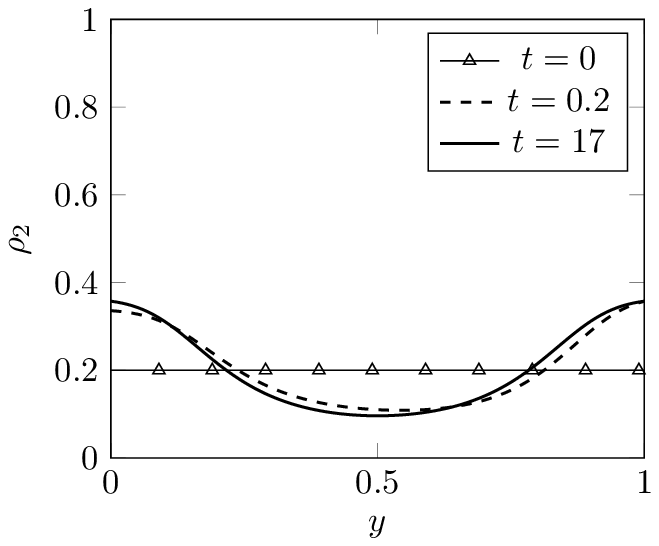}
\includegraphics[width=75mm]{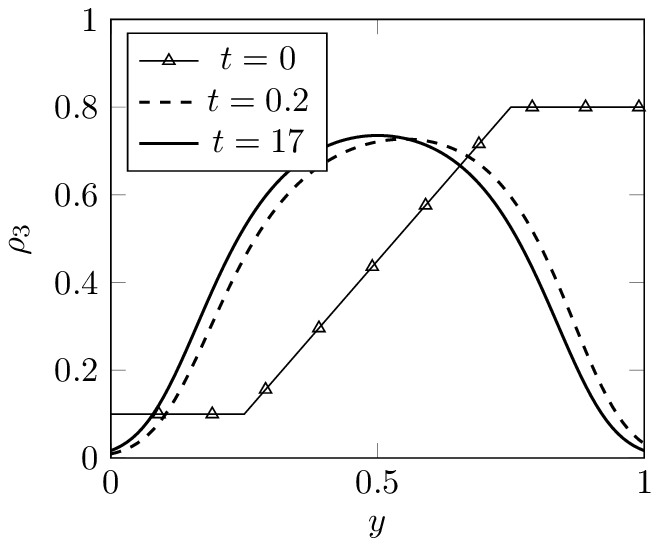}
\includegraphics[width=75mm]{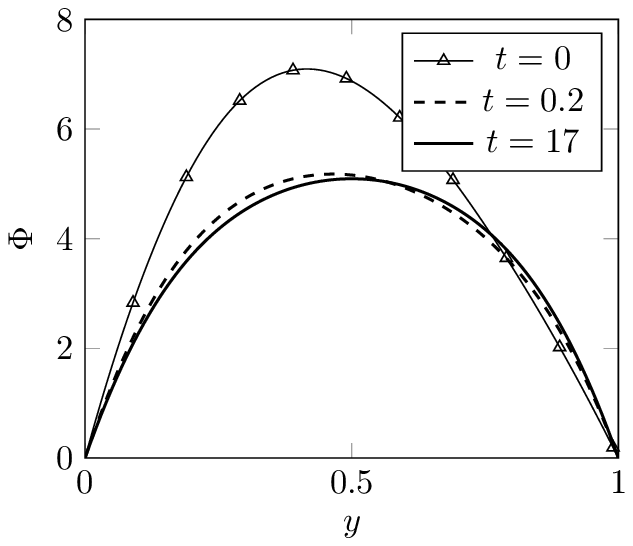}
\caption{Example 1: Particle densities $\rho_i$ and electric potential for molar masses
$M_1=M_2=M_3=1$ versus position at various times. 
The boundary conditions for the electric potential are in equilibrium.}
\label{fig.ex1}
\end{figure}

The situation changes
drastically when the molar masses are different (example 2). Figure \ref{fig.ex2} shows
the stationary solutions with the same parameters as in the previous example 
except $M_1=6$. Here, the discrete relative entropy is defined by 
$$
  H^*(\rho^k_h) = \int_0^1\bigg(c_{{\rm tot},h}^k\sum_{i=1}^n
	(x_h^k)_i\log\frac{(x_h^k)_i}{(x_h^\infty)_i}
	+ \frac{\lambda}{2}|\na(\Phi_h^k-\Phi_h^\infty)|^2\bigg)dy,
$$
where $(\rho_h^k,\Phi_h^k)$ is the finite-element solution at time $k\tau$
and $(x^\infty_h,\Phi_h^\infty)$ is the stationary solution. The integral
and gradients are computed by the trapezoidal and gradient routines of
MATLAB. The semi-logarithmic plot of the relative entropy shows that the entropy
converges to zero exponentially fast.

\begin{figure}[ht]
\includegraphics[width=75mm]{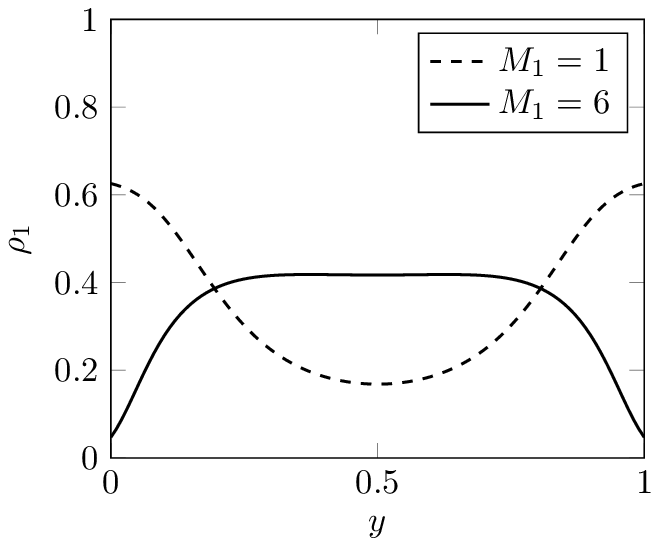}
\includegraphics[width=75mm]{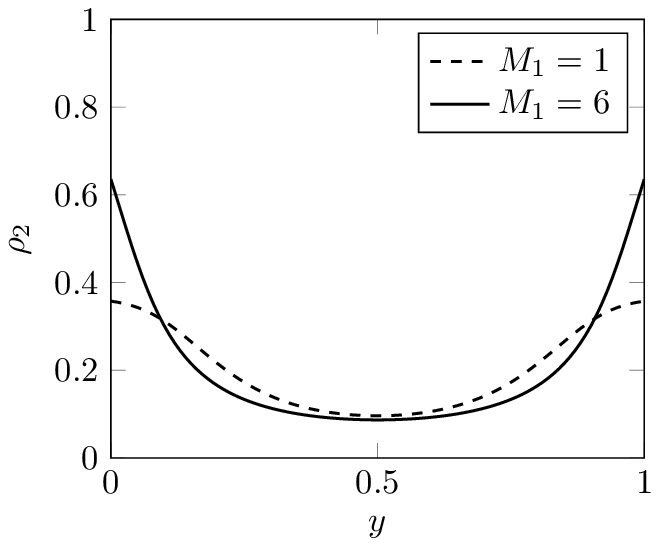}
\includegraphics[width=75mm]{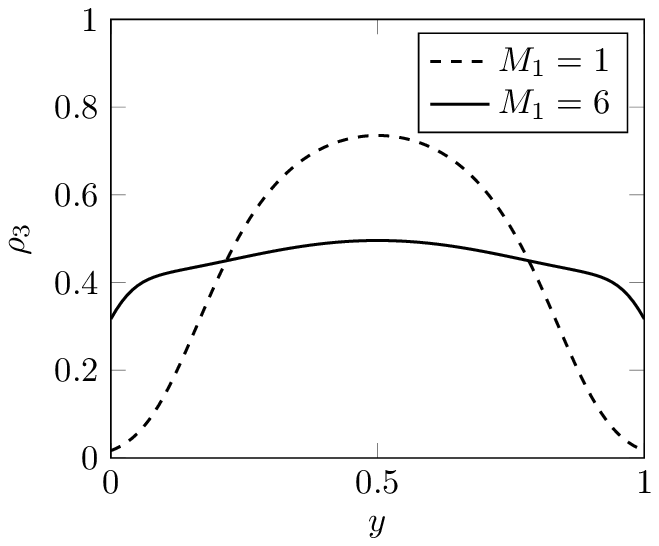}
\includegraphics[width=75mm]{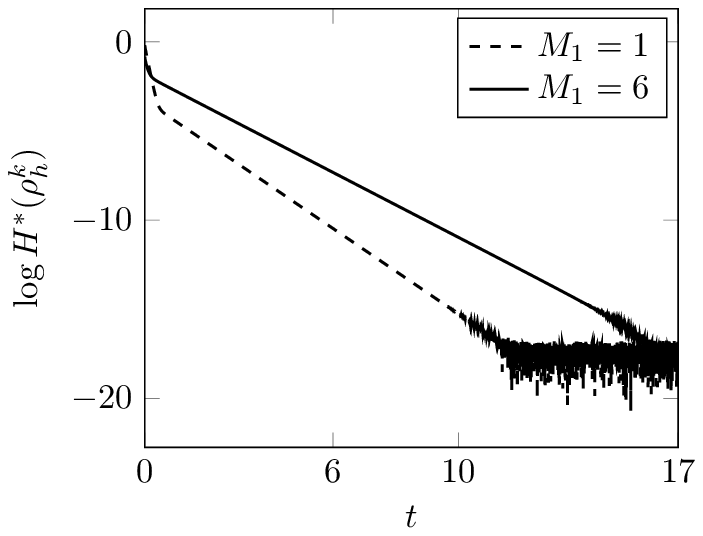}
\caption{Example 2: Particle densities $\rho_i$ at time $t=4$ versus position
and relative entropy (bottom right) for molar masses $M_1=6$ and $M_2=M_3=1$.
The boundary conditions for the electric potential are in equilibrium.}
\label{fig.ex2}
\end{figure}

For example 3, we choose the same initial conditions and parameters as before,
but we take non-equilibrium boundary data $\Phi(0)=10$, $\Phi(1)=0$. The
solutions at time $t=8$ for various molar masses $M_1$ are displayed in Figure
\ref{fig.ex3}. Since $\rho_1$ and $\rho_2$ have both positive charge and the
potential on the left boundary is positive, both species avoid the left
boundary and move to the right.

\begin{figure}[ht]
\includegraphics[width=75mm]{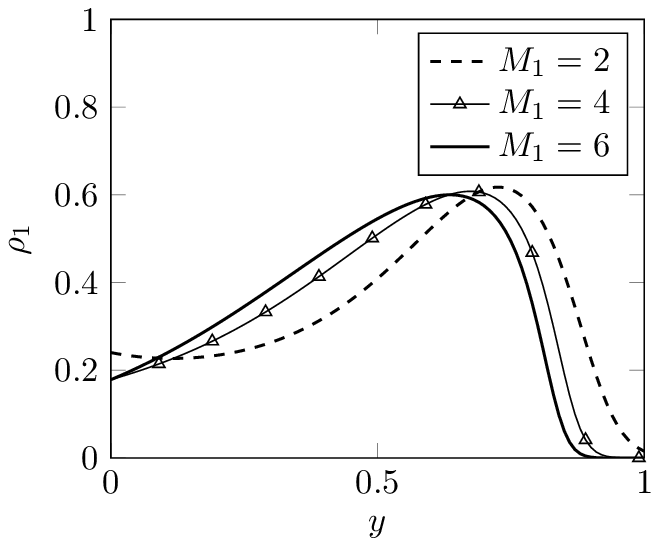}
\includegraphics[width=75mm]{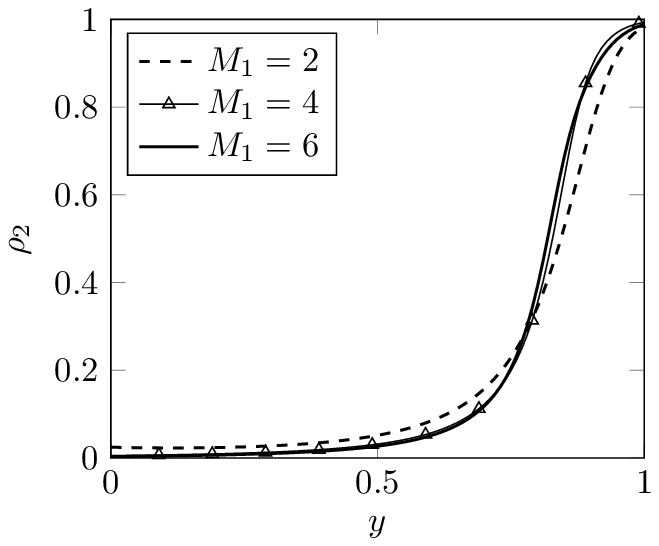}
\includegraphics[width=75mm]{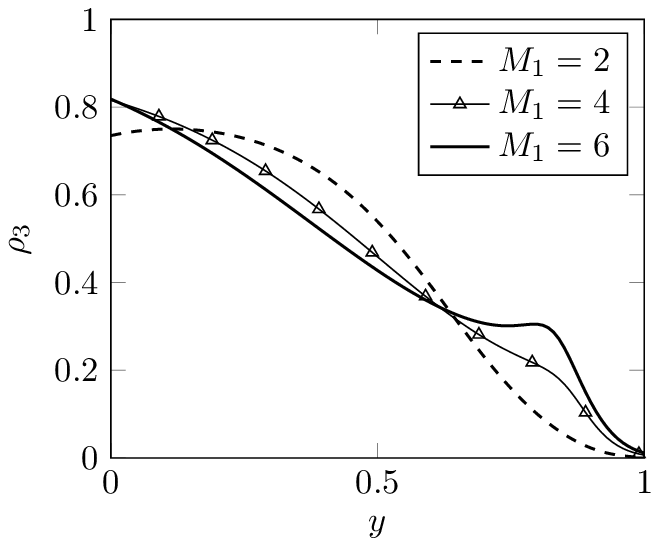}
\includegraphics[width=75mm]{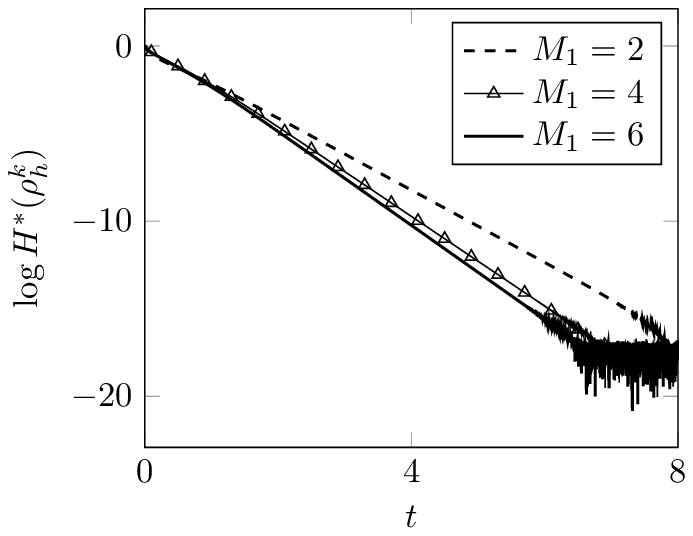}
\caption{Example 3: Particle densities $\rho_i$ at time $t=8$ versus position
and relative entropy (bottom right) for various molar masses $M_1$.
The boundary conditions for the electric potential are not in equilibrium.}
\label{fig.ex3}
\end{figure}

In example 4, we interchange the roles of $M_1$ and $M_2$, i.e.,
we choose $M_1=1$ and $M_2\in\{2,4,6\}$. We observe in Figure \ref{fig.ex4}
that the first species
is more concentrated at the right boundary while in the previous example,
this holds true for the second species.

\begin{figure}[ht]
\includegraphics[width=75mm]{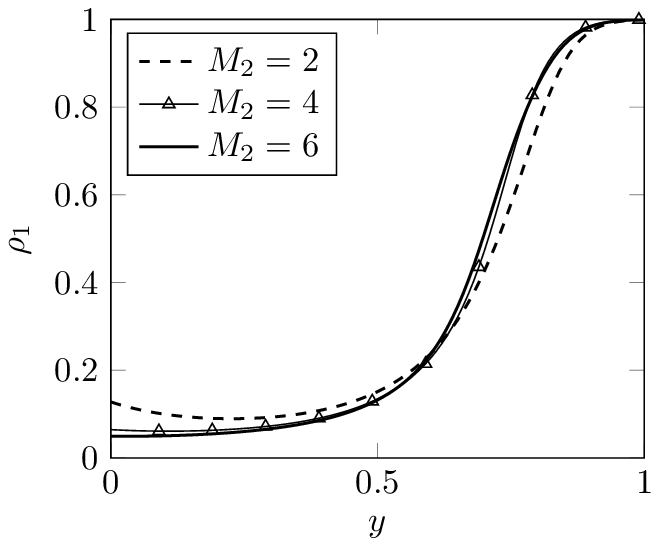}
\includegraphics[width=75mm]{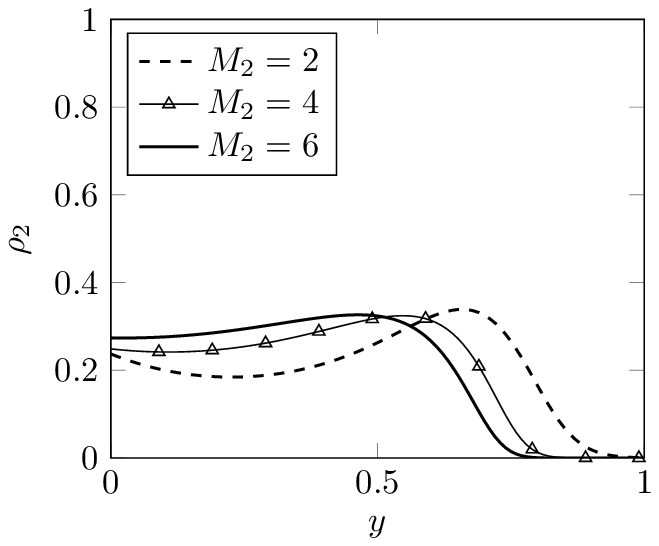}
\includegraphics[width=75mm]{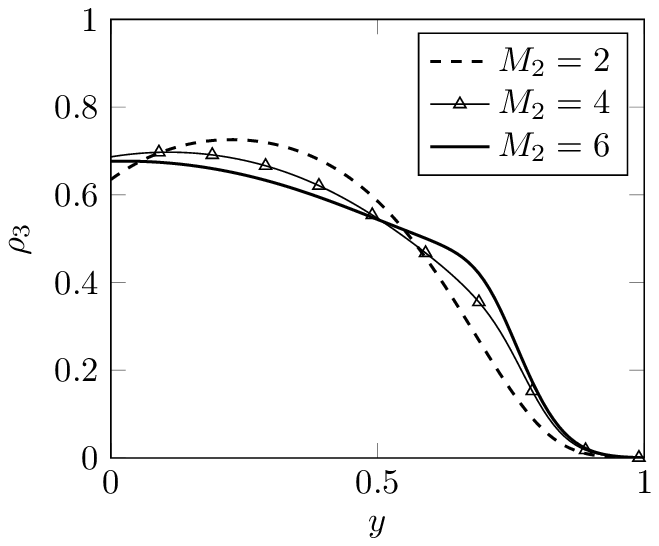}
\includegraphics[width=75mm]{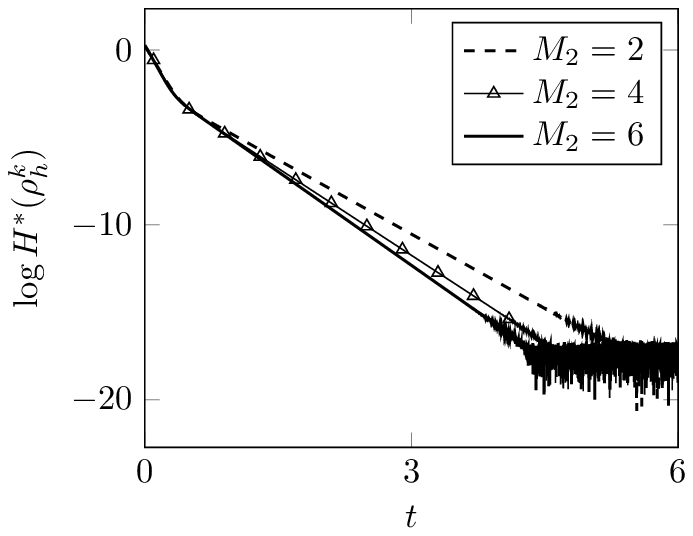}
\caption{Example 4: Particle densities $\rho_i$ at time $t=8$ versus position
and relative entropy (bottom right) for various molar masses $M_2$.
The boundary conditions for the electric potential are not in equilibrium.}
\label{fig.ex4}
\end{figure}

The previous examples show that the convergence rate to equilibrium
strongly depends on the ratio of the molar masses. It turns out that this
effect is triggered by the drift term, and without electric field, the
convergence rates are similar for different molar masses. This behavior
can be observed in Figure \ref{fig.ex5} (example 5), 
where we have taken the same parameters
as in the previous example but neglect the electric field. In this situation,
the steady state is constant in space and explicitly computable; indeed, we have
$\rho_i^\infty = \mbox{mean}(\Omega)^{-1}\|\rho_i^0\|_{L^1(\Omega)}$. Note
that the steady state in the previous examples is not constant.

\begin{figure}[ht]
\includegraphics[width=75mm]{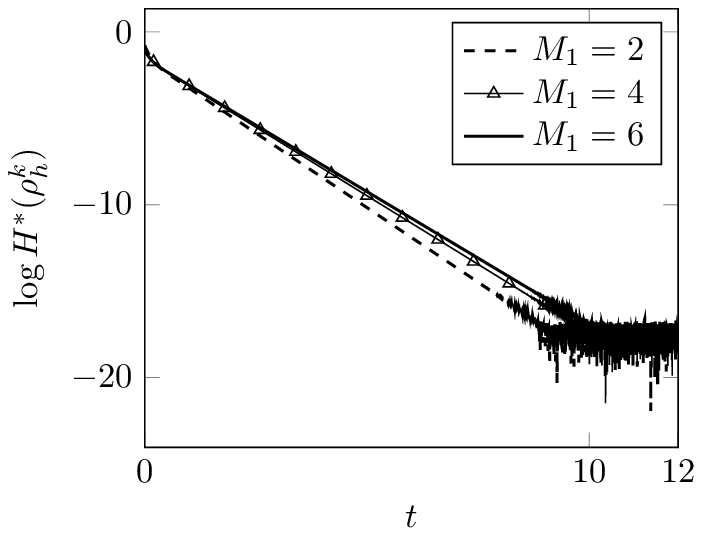}
\caption{Example 5: 
Semi-logarithmic plot of the relative entropy $H^*(\rho_h^k)$ versus time, 
without electric potential and for different molar masses.}
\label{fig.ex5}
\end{figure}

Finally, we compute the numerical convergence rate when the grid size tends
to zero for the situation of example 3 (non-equilibrium boundary conditions for the 
potential). We choose the time $t=0.01$ and the time step size $\tau=10^{-4}$. 
The solutions are computed on nested meshes with grid sizes
$h\in\{0.01, 0.005, 0.0025, 0.0006, 0.0001\}$ and compared to the reference
solution, computed on a very fine mesh with 25601 elements ($h\approx 4\cdot 10^{-5}$).
As expected, we observe a second-order convergence in space; see Figure
\ref{fig.conv}.

\begin{figure}[ht]
\includegraphics[width=75mm]{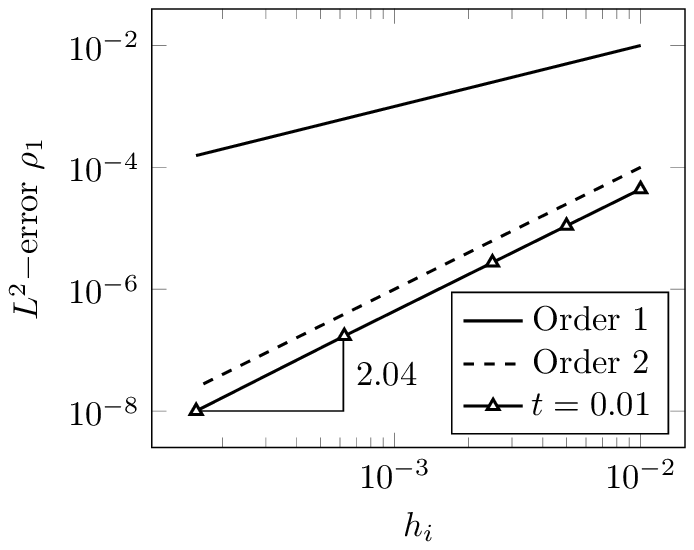}
\includegraphics[width=75mm]{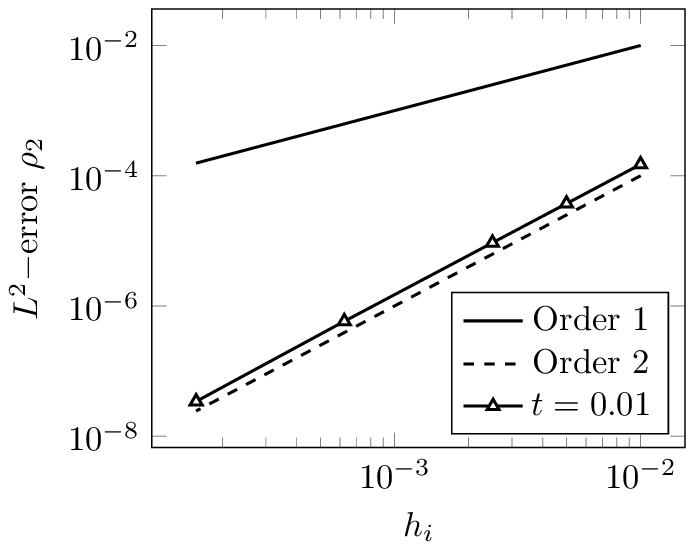}
\includegraphics[width=75mm]{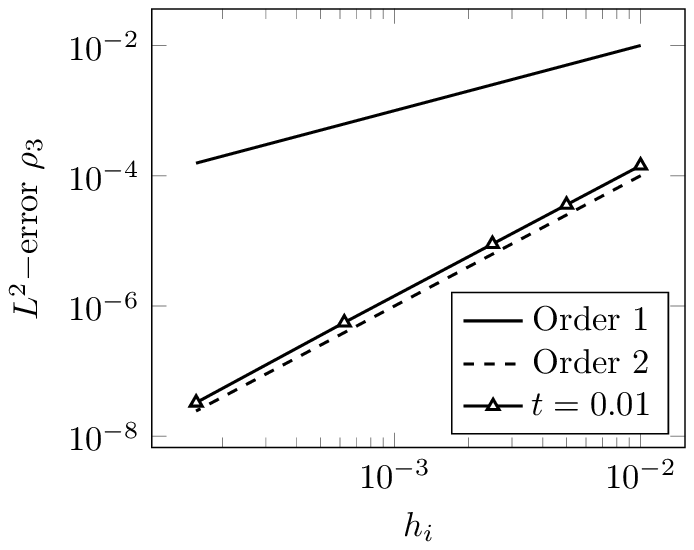}
\includegraphics[width=75mm]{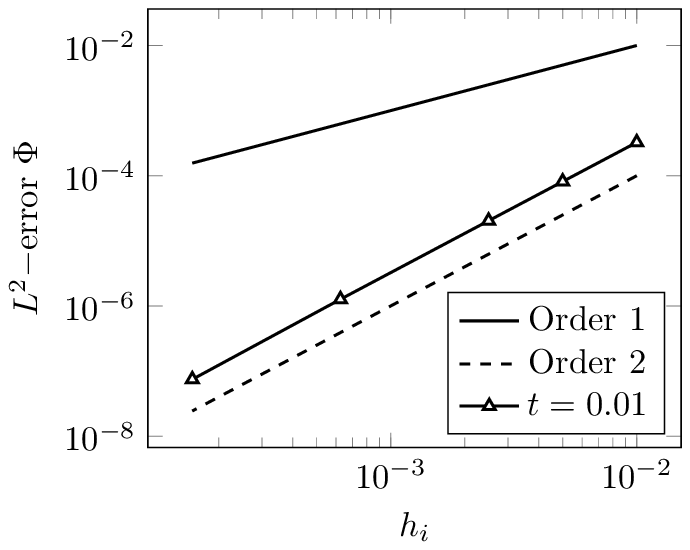}
\caption{Discrete $L^2$-error relative to the reference solution for the 
densities and the potential (bottom right) at time $t=0.01$.}
\label{fig.conv}
\end{figure}


\end{document}